\documentclass[12pt]{article}
\usepackage{amsmath,amsfonts,amssymb,amsthm}
\usepackage{mathrsfs}
\usepackage{latexsym}
\usepackage[english]{babel}
\usepackage[usenames,dvipsnames]{xcolor}
\usepackage{comment}
\usepackage{appendix}
\usepackage{comment} 
\usepackage{bbm}

\usepackage{graphicx}

\renewenvironment{proof}[1][\proofname]{{\bfseries #1.} }{\qed}

\newcommand{\field}[1]{\mathbb{#1}}

\newcommand{\lop}{\mathcal{L}}
\newcommand{\dom}{\mathsf{D}}
\newcommand{\R}{\field{R}}

\newcommand{\N}{\field{N}}

\newcommand{\Var}{{\rm Var}}

\newcommand{\mV}{{\mathcal V}}

\newcommand{\mH}{\mathcal{H}}

\def\authors#1{{ \begin{center} #1 \vspace{0pt} \end{center} } \smallskip}
\def\institution#1{{\sl \begin{center} #1 \vspace{0pt} \end{center} } }
\def\inst#1{\unskip $^{#1}$}
\def\title#1{{\huge\bf  \begin{center} #1 \vspace{0pt} \end{center}  } \smallskip}

\def\E{{\mathbb{ E}}}

\def\P{{\mathbb{P}}}

\def\tfrac#1#2{{\textstyle\frac {#1}{#2}}}
\newtheorem{theorem}{Theorem}[section]
\newtheorem{proposition}[theorem]{Proposition}
\newtheorem{lemma}[theorem]{Lemma}

\newtheorem{corollary}[theorem]{Corollary}
\newtheorem{definition}[theorem]{Definition}

\newtheorem{remark}[theorem]{Remark}

\begin{document}

\date{Sep 2020}

\title{\sc {\centering On the equivalence of Sobolev norms in Malliavin spaces}}
\authors{\large Davide Addona\footnote{E-mail address: \texttt{davide.addona@unipr.it}}, Matteo Muratori\footnote{E-mail address: \texttt{matteo.muratori@polimi.it}}, Maurizia Rossi\footnote{Corresponding author. E-mail address: \texttt{maurizia.rossi@unimib.it}.}
}
\institution{\inst{1}Dipartimento di Scienze Matematiche, Fisiche e Informatiche, Universit\`a di Parma\\
\inst{2}Dipartimento di Matematica, Politecnico di Milano\\
\inst{3}Dipartimento di Matematica e Applicazioni, Universit\`a di Milano-Bicocca 
}

\begin{abstract}

We investigate the problem of the equivalence of $L^q$-Sobolev norms in Malliavin spaces for $q\in [1,\infty)$, focusing on the graph norm of the $k$-th Malliavin derivative operator and the full Sobolev norm involving all derivatives up to order $k$, where $k$ is any positive integer. The case $q=1$ in the infinite-dimensional setting is challenging, since at such extreme the standard approach involving Meyer's inequalities fails. In this direction, we are able to establish the mentioned equivalence for $q=1$ and $k=2$ relying on a vector-valued Poincar\'e inequality that we prove independently and that turns out to be new at this level of generality, while for $q=1$ and $k>2$ the equivalence issue remains open, even if we obtain some functional estimates of independent interest. With our argument (that also resorts to the Wiener chaos) we are able to recover the case $q\in (1,\infty)$ in the infinite-dimensional setting; the latter is known since the eighties, however our proof is more direct than those existing in the literature, and allows to give explicit bounds on all the multiplying constants involved in the functional inequalities. Finally, we also deal with the finite-dimensional case for all $q\in [1,\infty)$ (where the equivalence, without explicit constants, follows from standard compactness arguments): our proof in such setting is much simpler, relying on Gaussian integration-by-parts formulas and an adaptation of Sobolev inequalities in Euclidean spaces, and it provides again quantitative bounds on the multiplying constants, which however blow up when the dimension diverges to $\infty$ (whence the need for a different approach in the infinite-dimensional setting). 

\smallskip


\smallskip


\end{abstract}




\section{Introduction}

We start by recalling the notion of isonormal Gaussian process along with the definition of the so-called Malliavin operators, and then introduce the problem of the equivalence of $L^q$-Sobolev norms in Malliavin spaces for $q\in [1,\infty)$, eventually describing our main results.

\subsection{Background and notation}\label{sec_back}

In the next subsections we give a brief overview of the main definitions and basic known results about Malliavin differential operators, that we will consistently use in the rest of the paper.  

\subsubsection{Isonormal Gaussian processes}

Let us fix a probability space $(\Omega, \mathcal F, \mathbb P)$, and denote by $\mathbb E$ the expectation under $\mathbb P$. All random objects in this paper are implicitly defined on the latter, unless otherwise specified. Let $\mH$ be a real separable Hilbert space with inner product $\langle \cdot, \cdot \rangle_{\mH}$, the corresponding norm being denoted by $\| \cdot \|_{\mathcal H}$. As we will deal with both finite-dimensional and infinite dimensional Hilbert space, we let ${\rm dim}(\mathcal H)$ stand for the dimension of $\mathcal H$.

\begin{definition}[Isonormal Gaussian process]
An isonormal Gaussian process $W$ over $(\mH, \langle \cdot, \cdot \rangle_{\mH})$ is
a real-valued centered Gaussian random field $
W \equiv \lbrace W(h), \, h \in \mH\rbrace
$
 such that for every $h,g\in \mH$ we have
\begin{equation}\label{covIso}
\mathbb E[W(h) W(g)]=\langle h, g\rangle_{\mH} \, .
\end{equation} 
\end{definition}
It is always possible to construct an isonormal Gaussian process over $(\mH, \langle \cdot, \cdot \rangle_{\mH})$ (see e.g.~\cite[Proposition 2.1.1]{NP12}), 
and plainly all isonormal Gaussian processes over $(\mH, \langle \cdot, \cdot \rangle_{\mH})$ share the same law. By definition, the map
\begin{equation*}
    h \mapsto W(h)
\end{equation*}
is linear, and in view of \eqref{covIso} it is easy to check that it is an isometry from $ (\mH,\langle \cdot , \cdot \rangle_\mH) $ onto a closed subspace of $L^2(\Omega)\equiv L^2(\Omega, \mathcal F, \mathbb P)$. In particular, if $\mathcal H$ is one-dimensional, then an isonormal Gaussian process $W$ over $(\mH, \langle \cdot, \cdot \rangle_{\mH})$ can be identified as the linear space generated by $N\sim \mathcal N(0,1)$, a standard Gaussian random variable. More in general, $ W $ can be identified as the closure in $ L^2(\Omega) $ of the linear space generated by $ \{ W(e_i) \}_{i \in \N} $, where $ \{ e_i \}_{i \in \N} $ is an orthonormal basis for $\mH$. For further details and a more complete account on isonormal Gaussian processes and related topics we refer the reader to \cite{NP12, Nua95}.

\subsubsection{Malliavin derivative operators}\label{sec_malliavin}

From here on we fix an isonormal Gaussian process $W$ over $(\mH, \langle \cdot, \cdot \rangle_{\mH})$, and we assume that $\mathcal F$ is the $\sigma$-field generated by the random variables $\lbrace W(h), \, h\in \mH\rbrace$. Note that, under this hypothesis, every random variable $ F : \Omega \to \mathbb{R} $  can be represented (not uniquely) by a function $ f : \mathbb{R}^{\mH} \to \R $, in the sense that
\begin{equation}\label{finite-dim-rv-pre}
F = f(W) \qquad \text{$\mathbb{P}$-a.s.~in } \Omega \, . 
\end{equation}
Most times it will be enough to deal with ``finite-dimensional'' random variables, that is, given $ m \in \N_{\ge 1} $ and $ h_1 , \ldots , h_m \in \mH $, in the case where $ F $ is measurable w.r.t.~the $ \sigma $-field generated by $ \{ W(h_i) \}_{i=1,\ldots,m} $ one can write
\begin{equation} \label{finite-dim-rv}
    F = f(W(h_1),\dots, W(h_m)) \, ,
\end{equation}
where now $ f: \mathbb{R}^m \to \mathbb{R} $ stands for a Borel-measurable function on $ \mathbb{R}^m $. Note that \eqref{finite-dim-rv} can be linked to \eqref{finite-dim-rv-pre} up to a projection onto the linear space spanned by $ \{ h_i \}_{i,\ldots,m} $.

Among finite-dimensional random variables a key role is taken by the space $\mathcal S$ of \emph{smooth random variables}, i.e., random variables of the form \eqref{finite-dim-rv} for some $ m \in \N_{\ge 1} $ and $ h_1 , \ldots , h_m \in \mH $, where $f$ is in addition a $ C^\infty(\mathbb{R}^m)$ function such that $f$ and all of its partial derivatives have at most polynomial growth (we let $ {S}(\mathbb{R}^m) $ denote the space of such functions). In fact $\mathcal S$ turns out to be dense in $L^q(\Omega) \equiv L^q(\Omega, \mathcal F, \mathbb P)$ for every $q\in [1,\infty)$, see e.g.~\cite[Lemma 2.3.1]{NP12}, so that $ L^q(\Omega) $ is always separable.

Since $ h \mapsto W(h) $ is a linear application, with no loss of generality one can assume that the vectors $ \{ h_i \}_{i,\ldots,m} $ in \eqref{finite-dim-rv} are orthonormal. We will however make this hypothesis explicit when needed. In particular, if ${\rm dim}(\mathcal H) = n < \infty $, it is enough to consider functions $ f: \mathbb{R}^n \to \mathbb{R} $.

\begin{definition}[Malliavin derivatives]
For $ k \in \mathbb N_{\ge1}$ the $k$-th Malliavin derivative of $F\in \mathcal S$ is the $\mH^{\otimes k}$-valued random variable defined as 
\begin{equation}\label{defDp}
    D^k F := \sum_{i_1,\dots,i_k=1}^m \frac{\partial f^k}{\partial x_{i_1} \cdots \partial x_{i_k}} (W(h_1),\dots, W(h_m)) \, h_{i_1}\otimes\cdots \otimes h_{i_k} \, .
\end{equation}
\end{definition}
Here and in what follows $\mH^{\otimes k}$ stands for the $k$-th tensor product of $\mH$, which is a Hilbert space if endowed with the standard scalar product $ \langle \cdot , \cdot \rangle_{\mH^{\otimes k}} $, and corresponding norm $ \| \cdot \|_{\mH^{\otimes k}} $, naturally induced by $ \langle \cdot , \cdot \rangle_{\mH} $. Actually, by Schwarz's theorem, $ D^k $ takes values in the closed subspace of $ k $-th \emph{symmetric} tensors, which is typically denoted by $ \mH^{\odot k} $. For $k=1$ we have of course $\mH^{\odot 1} = \mH^{\otimes 1} \equiv  \mH$, and we will often write $D$ (resp.~Malliavin derivative) in place of $D^1$ (resp.~first Malliavin derivative). Note that the Malliavin derivative of a smooth random variable represented by $f$ is nothing but the smooth random variable, with values in $ \mathbb{R}^m $ (seen as a subspace of $ \mH $), represented by the function $ \nabla f $. Moreover, it is not difficult to check that $ D^k $ is well defined, i.e., it does not depend on the chosen representative $f$ of $F$.

Given another real separable Hilbert space $(\mV, \langle \cdot, \cdot\rangle_{\mV})$ and $q\in [1,\infty)$, we let $L^q(\Omega;\mV)$ denote the space $L^q(\Omega, \mathcal F, \mathbb P;\mV)$ of $\mV$-valued random variables $ F $ such that $ \| F \|_{\mV} \in L^q(\Omega) $, that is $ \E[\| F \|_{\mV}^q] < +\infty $ and
$$
\left\|F \right\|_{L^q(\Omega;\mV)}:=\left( \E[\| F \|_{\mV}^q] \right)^{\frac 1 q}  .
$$
As a general rule, and when no ambiguity occurs, we will write $ \mathcal C( A ;\mV) $ to refer to the natural extension to $ \mV $-valued functions of a space $ \mathcal C(A) $ of real functions defined on a suitable set $A$. A typical choice for us will be $ \mV = \mH^{\otimes k} $. If $ \mathrm{dim}(\mathcal{V})=1 $ we will implicitly assume $ \mV \equiv \R $ and thus $ \mH^{\otimes k} \otimes \mV \equiv \mH^{\otimes k} $, without further mention, and adopt the above simplified notations.

The following key result holds.

\begin{lemma}[Proposition 2.3.4 in \cite{NP12}] \label{lem_closable}
For any $k \in \mathbb N_{\ge 1}$ and $q\in [1,\infty)$, the operator 
\begin{equation}\label{closable}
    D^k : \mathcal S\subset L^q(\Omega)  \to L^q(\Omega; \mH^{\otimes k})
\end{equation}
is closable in $L^q(\Omega)$.
\end{lemma}
Let us now consider the space $\mathbb D^{k,q}$, defined as the closure of $\mathcal S$ in $ L^q(\Omega) $ with respect to the \emph{full} Sobolev norm
\begin{equation}\label{normD}
    \| F\|_{\mathcal D(k,q)}^{q} := \| F\|_{L^q(\Omega)}^{q} +  \sum_{\ell=1}^k \left\| D^\ell F \right\|_{L^q(\Omega; \mH^{\otimes \ell})}^{q}.
\end{equation}
In view of Lemma \ref{lem_closable}, every operator $D^k$ in (\ref{closable}) can consistently be extended to the whole (Banach) space $\mathbb D^{k,q}$, hence we adopt the same symbol as in $ \mathcal{S} $ with no ambiguity. The space $\mathbb D^{k,q}$ is usually called the \emph{domain} of $D^k$ in $L^q(\Omega)$, see e.g.~\cite[\S 2.3]{NP12} and  \cite[\S 1.2]{Nua95}. Obviously for any $ \ell \in \mathbb N$ and any $\delta\in [0,+\infty)$ we have the inclusion
\begin{equation}\label{inclusion}
    \mathbb D^{k + \ell,q + \delta} \subseteq \mathbb D^{k,q} \, .
\end{equation}
In the light of these properties, we can interpret $ \mathbb D^{k,q} $ as a ``Sobolev space'' of random variables and give $ D^\ell $, $ \ell = 1,\ldots,k $, the role of weak derivatives.

\subsubsection{Malliavin derivatives in Hilbert spaces}\label{sec_malliavinH}

The content of the previous section can naturally be generalized to $ \mV $-valued functions, where $ (\mV, \langle \cdot,\cdot \rangle )_\mV$ is a real separable Hilbert space. For any $ k \in\N_{\ge 1}$ and any $q\in[1,\infty)$, it is still possible to define the $k$-th Malliavin derivative operator, which we keep denoting by $D^k$ with some abuse of notation, acting on the space of smooth $\mV$-valued random variables $\mathcal S_{\mV} \subset L^q(\Omega; \mV)$. The space $\mathcal S_{\mV}$ is understood as the collection of all random variables of the form $F=\sum_{j=1}^J F_j \, v_j$, where $J\in \mathbb N_{\ge 1}$, $F_j\in \mathcal S$ and $v_j\in \mV$ for every $ j=1 , \ldots , J $. Clearly, when needed, one can assume with no loss of generality that $ \{ v_j \}_{j=1,\ldots,J} $ are orthonormal vectors. Standard results, combining the density of $ \mathcal{S} $ in $ L^q(\Omega) $ and the separability of $ \mV $, ensure that $\mathcal S_{\mV}$ is dense in $L^q(\Omega;\mV)$ for any $q\in[1,\infty)$, so that also $ L^q(\Omega;\mV) $ is separable.
\begin{definition}[Malliavin derivatives in the $ \mV $-valued case]\label{defDp2}
For $ k \in \mathbb N_{\ge1}$ the $k$-th Malliavin derivative of $F \in \mathcal S_\mV$ is the $\mH^{\otimes k} \otimes \mV $-valued random variable defined as 
\begin{equation*}\label{defMall}
    D^k F := \sum_{j=1}^J  \left( D^k F_j \right) \otimes v_j \, ,
\end{equation*}
where $D^k F_j$ is given by (\ref{defDp}).
\end{definition}

Note that also $ \mH^{\otimes k} \otimes \mV $ is a Hilbert space if endowed with the natural scalar product $ \langle \cdot , \cdot \rangle_{\mH^{\otimes k} \otimes \mV } $ induced by $ \langle \cdot , \cdot \rangle_{\mH^{\otimes k} } $ and $\langle \cdot , \cdot \rangle_{ \mV } $. Plainly, we have $ D^k F \in \mathcal{S}_{\mH^{\otimes k} \otimes \mV} $. As above, will often write $D$ (resp.~Malliavin derivative) in place of $D^1$ (resp.~first Malliavin derivative).

The following result is the generalization of Lemma \ref{lem_closable} to the $ \mV $-valued case. 

\begin{lemma}[Proposition 2.5 in \cite{PV14}]\label{lem_gen}
For any $k\in \mathbb N_{\ge 1}$ and $q\in [1,\infty)$, the operator 
\begin{equation*}\label{closable2} 
    D^k : \mathcal S_{\mV}\subset L^q(\Omega; \mV)\to L^q(\Omega; \mathcal H^{\otimes k}\otimes \mV)
\end{equation*}
is closable in $L^q(\Omega;\mV)$.
\end{lemma}
Hence, as in the real-valued case, we can define the Sobolev space $\mathbb D^{k,q}(\mV)$ as the closure of $ \mathcal{S}_\mV $ in $L^q(\Omega;\mV)$ with respect to the Sobolev norm
\begin{equation*}\label{defD2}
    \| F \|_{\mathcal D^{k,q}(\mV)}^{q} : = \| F\|_{L^q(\Omega;\mV)} ^{q}+ \sum_{\ell=1}^k \left\| D^\ell F \right\|_{L^q(\Omega; \mathcal H^{\otimes \ell}\otimes \mV)}^{q} ,
\end{equation*} 
and consistently extend $ D^k $ to the whole $\mathbb D^{k,q}(\mV)$. It is obvious that relations analogous to (\ref{inclusion}) hold, and it is readily seen that for every $ F \in \mathbb D^{k+\ell,q}(\mV)$ (let $  \ell \in \N_{\ge 1} $) we have $ D^\ell F \in \mathbb{D}^{k,q}(\mH^{\otimes \ell} \otimes \mathcal{V}) $ and
$$
D^{k+\ell} F = D^k\!\left( D^\ell F \right) .
$$
If $\mV \equiv  \mathbb R$, in order to be coherent with the previous notations, we will always drop the dependence on $\mV$, for instance we will write $\mathbb D^{k,q}$ (resp.~$\mathcal S$) instead of $\mathbb D^{k,q}(\mathbb R)$ (resp.~$\mathcal S_{\mathbb R}$), and so forth.


\subsection{Motivations, equivalent problems and description of the main results}\label{subsec_mot}

In view of Lemma \ref{lem_gen}, it seems more natural to extend the Malliavin derivative $ D^k $ to the (a priori larger) space $
 \mathbb{D}_\ast^{k,q}(\mV)$ defined as the closure in $L^q(\Omega;\mV)$ of $\mathcal S_\mV$ with respect to the (weaker) norm 
\begin{equation}\label{normG}
    \| F \|_{\mathcal G(k,q)(\mV)} := \| F\|_{L^q(\Omega;\mV)} + \left\| D^k F \right\|_{L^q(\Omega; \mH^{\otimes k}\otimes \mV)} ,
\end{equation}
{which is usually referred to as the {\it graph norm} of $D^k$ on $L^q(\Omega;\mV)$}. Nevertheless, the extension to $\mathbb D^{k,q}(\mV)$ described in \S \ref{sec_malliavin} is much more convenient, for instance it allows for an ordering of Malliavin domains, recalling (\ref{inclusion}) and the corresponding generalization in \S \ref{sec_malliavinH}. Clearly for $ k=1 $ we have $ \mathbb{D}^{1,q}(\mV) = \mathbb{D}^{1,q}_\ast(\mV)  $,  whereas for $ k \ge 2 $, a priori, we only have the trivial inequality 
\begin{equation}\label{ineq1}
    \| F\|_{\mathcal G(k,q)(\mV)}\le  {2^{1-\frac 1 q}}\|F\|_{\mathcal D(k,q)(\mV)} \quad \forall F \in \mathcal{S}_\mV 
\end{equation}
and thus the inclusion $ \mathbb{D}^{k,q}(\mV) \subseteq \mathbb{D}^{k,q}_\ast(\mV)  $. Hence a natural question is whether, up to constants, also the reverse inequality holds, i.e., the norms $\|\cdot \|_{\mathcal G(k,q)(\mV)}$ and $\| \cdot\|_{\mathcal D(k,q)(\mV)}$ are equivalent on $  \mathcal{S}_\mV  $. If so, then the spaces $\mathbb D^{k,q}_\ast(\mV)$ and $ \mathbb{D}^{k,q}(\mV)$ would coincide as well as the extensions of the operator $D^k$ to these domains.

A fact worth observing is that, actually, the equivalence of norms is also a \emph{necessary} condition for the spaces $	\mathbb{D}_\ast^{k,q}(\mV)$ and $
	\mathbb{D}^{k,q}(\mV)$ to coincide. Indeed, thanks to Lemma \ref{lem_gen}, both $ \mathbb{D}^{k,q}_\ast(\mV) $ and  $ \mathbb{D}^{k,q}(\mV)  $ are Banach spaces when equipped with the norms $  \| \cdot \|_{\mathcal G(k,q)(\mV)} $ and $  \| \cdot \|_{\mathcal D(k,q)(\mV)} $, respectively. Therefore, if these spaces coincide, we deduce that $ \mathbb{D}^{k,q}(\mV)   $ is a Banach space with respect to both such norms; moreover, due to \eqref{ineq1}, by density we have that $  \| \cdot \|_{\mathcal G(k,q)(\mV)}  \le 2^{1-1/q}  \| \cdot \|_{\mathcal D(k,q)(\mV)}  $ on the whole $ \mathbb{D}^{k,q}(\mV) $. However, a well-known consequence of the open mapping theorem guarantees that if two norms make the same linear space a Banach space, then either they are equivalent or they are not ordered (up to constants). As a result, in this case $  \| \cdot \|_{\mathcal G(k,q)(\mV)} $ and $  \| \cdot \|_{\mathcal D(k,q)(\mV)} $ must necessarily be equivalent on $ \mathbb{D}^{k,q}(\mV) $, in particular on $ \mathcal{S}_\mV $.

Returning to our question, if $\mathcal H$ is a \emph{finite}-dimensional Hilbert space the answer is positive for any $k \in \N_{\ge 1} $ and any $ q\in [1,\infty) $: this is proved in Theorem \ref{thm3}. The proof we provide is elementary and adapts one-dimensional Sobolev-type inequalities from \cite[Chapter 5]{AF03} to the case of the standard Gaussian measure. However, such technique does not apply when $\mathcal H$ has infinite dimension, since the constants which appear in the computations explicitly depend on  $ n = \mathrm{dim} (\mH)  $ and blow up as $n\rightarrow \infty$.

When $ \mH $ is an \emph{infinite}-dimensional Hilbert space, we prove the equivalence of norms for any $k \in \N_{\ge 2} $ and any $ q\in (1,\infty) $ (Theorem \ref{thm2}) by means of a completely different method, which takes advantage of a \emph{Poincar\'e inequality}, that we establish independently (Theorem \ref{prop_poincare}), plus standard properties of the projection operator on the $ k $-th Wiener chaos (see Appendix \ref{appendixA}). For $ q=1 $, still using the Poincar\'e inequality combined with the one-dimensional technique, we establish the equivalence of norms for $ k=2 $ (Theorem \ref{th1}), the case of a general $ k\ge 3 $ being left open.

We point out that, for $ q>1 $, most of the results we present were already known or could be deduced from previous ones. However, in such case, our purpose is to show the equivalence inequalities by means of simpler arguments, allowing for \emph{quantitative constants} that can be computed explicitly. On the contrary, in the case $ q=1 $ most of our results are new. For a more complete discussion, with an accent on the proof strategies, we refer to \S \ref{sec_previous}.

\subsection*{Acknowledgments} 

D.A. has financially been supported by the Programme ``FIL-Quota Incentivante" of University of Parma and co-sponsored by Fondazione Cariparma, and by the INdAM-GNAMPA Project 2020 \emph{Elliptic Operators with Unbounded and Singular Coefficients in Weighted $L^p$ Spaces} (Italy). M.M. has been supported by the PRIN Project 2017 \emph{Direct and Inverse Pro\-blems for Partial Differential
Equations: Theoretical Aspects and Applications} (Italy). M.R. has been supported by the ANR-17-CE40-0008 Project \emph{Unirandom} (France) and the INdAM-GNAMPA Project 2020 \emph{Geometria Stocastica e Campi Aleatori} (Italy).  All the authors thank the INdAM-GNAMPA Research Group (Italy). 

\section{Main results and outline of the paper} 

In this section, we first state the main results of our paper regarding the equivalence of Malliavin norms, then briefly describe our strategy along with some auxiliary results needed, and finally compare with what was known in the literature.


\subsection{Statement of the main results}

Recall the content of \S \ref{sec_back}, in particular the discussion in \S \ref{subsec_mot}.  Our first main result concerns the equivalence of the norms $\| \cdot\|_{\mathcal G(k,q)(\mV)}$ and $\|\cdot \|_{\mathcal D(k,q)(\mV)}$, defined in  \eqref{normD} and \eqref{normG}, respectively, for any $q\in (1,\infty)$, any $k \in \mathbb N_{\ge 2}$ and any separable Hilbert spaces $ (\mH,\langle \cdot , \cdot \rangle_{\mH})  $ and $(\mathcal V, \langle \cdot, \cdot \rangle_{\mathcal V})$ (in particular they can both be infinite dimensional). Partial results for the critical exponent $ q=1 $ are discussed separately. Finally, we deal with the finite-dimensional case, i.e., $ \mathrm{dim}(\mH) < \infty $, where a more complete picture is available.

For the sake of clarity, we will state all the main results for functions in $ \mathcal{S}_\mV $, but note that they can readily be extended by density to the appropriate Sobolev space $\mathbb{D}^{k,q}(\mathcal{V}) $.

\subsubsection{The infinite-dimensional case for $\boldsymbol{q\in (1,\infty)}$}

\begin{theorem}
\label{thm2}
Let $(\mathcal H, \langle \cdot, \cdot \rangle_{\mathcal H})$ and $(\mathcal V, \langle \cdot, \cdot \rangle_{\mathcal V})$ be any real separable Hilbert spaces, $ k \in \N_{\ge 2} $ and $ q \in (1,\infty) $. Then the norms $\|\cdot\|_{\mathcal G(k,q)(\mathcal V)}$ and $ \|\cdot\|_{\mathcal D(k,q)(\mathcal V)}$ are equivalent. More precisely, for all $F\in \mathcal S_{\mathcal V}$, we have
\begin{equation}\label{estimate_fdc}
{2^{\frac 1q-1}}  \| F \|_{\mathcal G(k,q)(\mathcal V)} \le \| F\|_{\mathcal D(k,q)(\mathcal V)} \le \tau_{k,q} \, \| F \|_{\mathcal G(k,q)(\mathcal V)} \, ,
\end{equation}
where 
\begin{equation}\label{const}
\tau_{k,q} := \prod_{\ell=1}^{k-1} \left( 1+ d_{\ell,q} \vee c_q \right) ,
\end{equation}
the positive constants $ d_{\ell,q} $ and $c_q$ being defined as in \eqref{const_exp} and \eqref{cost_poincare}, respectively.
\end{theorem}

As mentioned in \S \ref{subsec_mot}, the norm equivalence stated above is known, see e.g.~\cite[Theorem 2.4]{Sug85}, \cite[Theorem 4.6]{Shi98} or \cite[Corollary 3.2]{PV14}, where the authors actually prove it in the more general framework of the so-called \emph{UMD} Banach spaces. However, here we carry out an alternative and more direct proof that works in separable Hilbert spaces, which allows us to provide \emph{quantitative constants} as in \eqref{const}. 

Unfortunately, the technique we develop cannot fruitfully be employed to cover the exponent $q=1$ as well: a different approach is hence required. 

\subsubsection{The infinite-dimensional case for $ \boldsymbol{q=1} $}

\begin{theorem}\label{th1}
	Let $(\mathcal H, \langle \cdot, \cdot \rangle_{\mathcal H})$ and $(\mathcal V, \langle \cdot, \cdot \rangle_{\mathcal V})$ be any real separable Hilbert spaces and $ \ell \in \N_{\ge 1} $. Then, for all $F\in \mathcal S_{\mV}$, we have
\begin{equation}
\label{stima_principale_1}
\left\| D^\ell F \right\|_{L^1(\Omega;\mathcal H^{\otimes \ell}\otimes \mV)} \leq \eta \left( \left\| D^{\ell-1}F \right\|_{L^1(\Omega;\mathcal H^{\otimes \ell-1}\otimes \mV)} +  \left\| D^{\ell+1}F \right\|_{L^1(\Omega;\mathcal H^{\otimes \ell+1}\otimes \mV)}  \right) ,
\end{equation}
where we set $ D^0 F := F $, $ \mH^{\otimes 0} \otimes \mV := \mV $ and
\begin{equation}\label{eta}
\eta := \frac\pi2+18\sqrt{2e} \, .
\end{equation}
\end{theorem}
To the best of our knowledge Theorem \ref{th1} is completely new. Remarkably, it implies the equivalence of Malliavin norms for $q=1$ and $k=2$, a first important step towards the solution of the open problem raised in \S \ref{subsec_mot}.
However, it is not possible to iterate \eqref{stima_principale_1} in order to get the equivalence of the norms $\|\cdot\|_{\mathcal G(k,1)(\mathcal V)}$ and $\|\cdot\|_{\mathcal D(k,1)(\mathcal V)}$. Nevertheless, by applying it to either odd-order or even-order derivatives, one can easily deduce the following partial equivalence result.

\begin{corollary}\label{coro_principale} 
Let $(\mathcal H, \langle \cdot, \cdot \rangle_{\mathcal H})$ and $(\mathcal V, \langle \cdot, \cdot \rangle_{\mathcal V})$ be any real separable Hilbert spaces, and let the constant $ \eta $ be defined as in \eqref{eta}. Then:
\begin{itemize}
    \item[(i)] The norms $\|\cdot\|_{\mathcal G(2,1)(\mathcal V)}$ and $\|\cdot\|_{\mathcal D(2,1)(\mathcal V)}$ are equivalent. More precisely, for all $F\in \mathcal S_{\mathcal V}$, we have
    \begin{equation*}
        \| F \|_{\mathcal G(2,1)(\mV)} \le \| F\|_{\mathcal D(2,1)(\mV)} \le (1+\eta) \, \| F\|_{\mathcal G(2,1)(\mV)} \, .
    \end{equation*}

\item[(ii)] For any $k \in \mathbb N_{\ge 3}$, and all $ F \in \mathcal{S}_\mV $, we have
\begin{equation}\label{aa}
\|F\|_{\mathcal D(k,1)(\mathcal V)} \leq \left(1 + 2 \eta \right) \left(\|F\|_{\mathcal G(k,1)(\mathcal V)} +\sum_{\ell=1}^{ \lceil k/2 \rceil -1 } \left\| D^{2\ell} F \right\|_{L^1(\Omega;\mathcal H^{\otimes 2\ell}\otimes \mV)}\right) ,
\end{equation}
\begin{equation}\label{bb}
\|F\|_{\mathcal D(k,1)}
\leq \left(1 + 2 \eta \right) \left (\|F\|_{\mathcal G(k,1)(\mV)} +\sum_{\ell=1}^{ \lfloor  k/2 \rfloor} \left\| D^{2\ell-1}F \right\|_{L^1(\Omega;\mathcal H^{\otimes 2\ell-1}\otimes \mV)}\right ) .
\end{equation}
\end{itemize}
\end{corollary}
In particular, by $(ii)$ we can infer that the norm $\|\cdot \|_{\mathcal D(k,1)}$ is equivalent to the norm $\|\cdot\|_{\mathcal G(k,1)}$ plus the $ L^1 $ norm of intermediate derivatives, of either even or odd order. In infinite dimension, i.e., when $ \mathrm{dim}(\mathcal{H}) = \infty $, we are not able to further improve $(ii)$, so that the question of whether the norms $\|\cdot \|_{\mathcal G(k,1)(\mathcal V)}$ and $\|\cdot\|_{\mathcal D(k,1)(\mathcal V)}$ are equivalent remains open for $k \in \mathbb N_{\ge 3}$, even in the case $\mathcal V \equiv \mathbb R$. On the other hand, as we will see in a moment, when $ \mathrm{dim}(\mathcal{H}) < \infty $ we have full equivalence for any $ q \in [1,\infty) $ and $ k \in \N_{\ge 2} $. 

It is worth stressing that, while proving Theorem \ref{thm2} and Theorem \ref{th1}, a key role is taken by the \emph{Poincar\'e inequality} (see Theorem \ref{prop_poincare}), that we prove independently, and estimates for \emph{expected} Malliavin derivatives (see Lemma \ref{prop2fd} and Lemma \ref{prop2}). For more detail we refer to \S \ref{sub_outline}.

\subsubsection{The finite-dimensional case for $ \boldsymbol{ q \in [1,\infty) } $}

\begin{theorem} \label{thm3}
Let $(\mathcal H, \langle \cdot, \cdot \rangle_{\mathcal H})$ and $(\mathcal V, \langle \cdot, \cdot \rangle_{\mathcal V})$ be any real Hilbert spaces, with $ \mH $ finite-dimensional and $ \mV $ separable. Let $ k \in \N_{\ge 2} $, $ q \in [1,\infty) $  and $ n:= \mathrm{dim}(\mH)  $. Then the norms $\|\cdot\|_{\mathcal G(k,q)(\mathcal V)}$ and $ \|\cdot\|_{\mathcal D(k,q)(\mathcal V)}$ are equivalent. More precisely, for all $F\in \mathcal S_{\mathcal V}$, we have
	\begin{equation}\label{estimate_fdc3}
	\begin{gathered}
  	{2^{\frac 1 q-1}}\| F \|_{\mathcal G(k,q)(\mathcal V)}\! \le\! \| F\|_{\mathcal D(k,q)(\mathcal V)} 
	\le  \frac{2^{k+1} \, C_{k-1,n}^{k}}{\varepsilon^{k-1}} \, \| F \|_{L^q(\Omega;\mV)} + \left( 1 + \varepsilon \right) \left\| D^k F \right\|_{L^q(\Omega;\mH^{\otimes k} \otimes \mV )} \\ \forall \varepsilon \in (0,1) \, ,
	\end{gathered}
	\end{equation}
	the positive constants $ C_{k,n} $ being defined as in \eqref{zeta}. 
\end{theorem}
As mentioned in \S \ref{subsec_mot}, the norm equivalence stated in Theorem \ref{thm3} is obtained through a direct proof inspired by the one-dimensional results of \cite[Chapter 5]{AF03}, thanks to which we can provide quantitative constants as in \eqref{estimate_fdc3} plus a further degree of freedom given by the parameter $ \varepsilon $, which allows one to reduce as much as possible the dependence of the norms of the previous derivatives on the norm of the $ k $-th derivative. Note that such a property cannot be achieved when using Poincar\'e inequalities, a tool that we do not exploit at all in finite dimension. However, Theorem \ref{thm3} does not generalize to the infinite-dimensional case, since the dimension $n$ of the space $\mH$ explicitly appears in the constants \eqref{zeta}, and the latter blow up as $n\rightarrow \infty$.

\begin{remark}[About compact embeddings]\label{rem-comp} \rm
A typical way of proving inequalities like \eqref{estimate_fdc3} is by means of \emph{compactness}. Indeed, it was proved in \cite{Hoo81} that for any  $ q \in [1,\infty) $ the Sobolev space $ W^{1,q}(\R^n ,\gamma_n) $ is compactly embedded into $ L^q(\mathbb{R}^n ,\gamma_n) $, where $ \gamma_n $ stands for the standard Gaussian measure on $ \R^n $. This readily yields the compactness of the embedding of $ W^{k,q}(\R^n ,\gamma_n) $ into $ W^{k-1,q}(\R^n,\gamma_n) $ for any $ k \ge 2 $, and the latter is in turn equivalent to the compactness of the embedding of $ \mathbb{D}^{k,q}(\mV) $ into $ \mathbb{D}^{k-1,q}(\mV) $ whenever both $ \mH  $ and $ \mV $ are finite dimensional. Then, a routine argument by contradiction ensures that for every $ \varepsilon>0 $ there exists a constant $ C_\varepsilon>0 $ such that
	$$
\| F \|_{\mathcal{D}(k-1,q)(\mV)}	\le  C_\varepsilon \, \| F \|_{L^q(\Omega;\mV)} + \varepsilon \, \| F \|_{\mathcal{D}(k,q)(\mV)}  \qquad \forall F \in \mathbb{D}^{k,q}(\mV) \, ,
$$
which yields an equivalence inequality analogous to \eqref{estimate_fdc3}. However, such an argument has some essential drawbacks. Firstly, the constant $ C_\varepsilon $ being obtained by contradiction, there is no way to make it quantitative. Secondly, it strongly relies on the finite dimension of \emph{both} $ \mH $ \emph{and} $ \mV $. Indeed, if $ \mV $ has infinite dimension then any sequence of constant functions $ \{ F_i \} \equiv \{ v_i \} $ is bounded in any $ \mathbb{D}^{k,q}(\mV) $ space, but if $ \{ v_i \} \subset \mV $ does not admit strongly convergent subsequences in $ \mV $, the same holds at the level of sequences in $ L^q(\Omega;\mV)$. If $ \mH $ has infinite dimension, then the sequence 
$$
F_i := W(h_i) \, , 
$$
where $ \{ h_i \}_{i \in \N} $ are orthonormal vectors in $ \mH $, converges weakly to $ 0 $ in $ L^q(\Omega) $ for any $ q \in [1,\infty) $ but is clearly bounded in every $ \mathbb{D}^{k,q} $ space and has constant nonzero $ L^q(\Omega) $ norm. Therefore, in both cases the compactness of the embeddings necessarily fails.
\end{remark}

\subsection{Outline of the paper and Poincar\'e inequality}\label{sub_outline}

The starting point in order to prove Theorems \ref{thm2} and \ref{th1} is to write, for all  $\ell=1,\dots,k-1$, $ q \in [1,\infty) $ and $F\in \mathcal S_\mV$,
\begin{equation}\label{app1}
    \left\| D^\ell F \right\|_{L^q(\Omega;\mH^{\otimes \ell}\otimes \mV)} \le \left\| D^\ell F -\mathbb E \!\left[D^\ell F\right] \right\|_{L^q(\Omega;\mH^{\otimes \ell}\otimes \mV)}
 + \left\| \mathbb E \!\left[D^\ell F\right] \right\|_{\mH^{\otimes \ell } \otimes \mV}  .
\end{equation}
We are going to treat the two addends on the right-hand side of (\ref{app1}) separately (see Theorem \ref{prop_poincare} and Lemmas \ref{prop2fd}, \ref{prop2}). As for Theorem \ref{thm3}, the approach is slightly different and entirely relies on Lemma \ref{lem1fd}. 

\subsubsection{Poincar\'e inequality}\label{sub_poincare}

Let us deal with the first addend on the right-hand side of (\ref{app1}). To this end, the following \emph{Poincar\'e inequality}, that we will prove in \S \ref{sec_poincare}, is crucial.

\begin{theorem}\label{prop_poincare}
Let $(\mathcal H, \langle \cdot, \cdot \rangle_{\mathcal H})$ and $(\mathcal V, \langle \cdot, \cdot \rangle_{\mathcal V})$ be any real separable Hilbert spaces and $ q \in [1,\infty) $. Then, for all $F\in \mathcal S_{\mathcal V}$, we have	
\begin{equation}\label{eq_poincare1}
    \left\| F -\mathbb E[F] \right\|_{L^q(\Omega; \mV)} \le c_{q} \left\| D  F \right\|_{L^q(\Omega;\mH \otimes \mV)} ,
\end{equation}
where 
\begin{equation}
\label{cost_poincare}
c_q:=
\begin{cases}
\sqrt{q-1} & \text{for } q\in [2,\infty) \, , \\
 \tfrac \pi 2  & \text{for } q\in[1,2) \, .
\end{cases}
\end{equation}
\end{theorem}
To the best of our knowledge Theorem \ref{prop_poincare}, in its generality, is \emph{new} for $q=1$. In a similar setting (actually for functions with values in suitable Banach spaces, as recalled above), the authors of \cite{PV14} proved the same Poincar\'e inequality for $q\in (1,\infty)$, with however no quantitative information on the multiplying constant. The proof of \cite[Proposition 3.1]{PV14} generalizes the scalar-valued proof given in \cite[Proposition 1.5.8]{Nua95} and strongly relies on Meyer's inequalities (see also the discussion in \S \ref{sec_previous}); our argument, instead, takes advantage of the definition of the Ornstein-Uhlenbeck semigroup in $L^2(\Omega;\mV)$ by means of bilinear forms and, for $q\in[1,2)$, on a corresponding duality argument. This allows us to reach the case $q=1$, left unsolved by such previous works.
We point out that, for \emph{even} $q \in [2,\infty) $ and $\mathcal V \equiv \mathbb R$, inequality \eqref{eq_poincare1} was proved in \cite{NPR09} with exactly the same constant $c_q=\sqrt{q-1}$. 

It is worth recalling that the first Poincar\'e inequality, in the special case $ \mH =\mV=\R $ and $ q=2 $, dates back to Nash \cite{Nas58}. The result was then rediscovered by Chernoff \cite{Che81} (both proofs use Hermite polynomials, see \cite[\S 1.4]{NP12}). More general, infinite-dimensional (i.e.~$ \mathrm{dim}(\mH)=\infty $) versions of the Poincar\'e inequalities were established by Houdr\'e and P\'erez-Abreu \cite{HP95}, for $q=2$, and subsequently by Milman \cite{Mil09}, for log-concave measures and any $ q \in [1,\infty) $, but scalar-valued functions (i.e.~$ \mV =\R $). {As concerns abstract Wiener settings, in \cite{FeUs00} the authors showed an $ L^2 $ Poincar\'e inequality for weighted Gaussian measures, while in \cite{VN15} the author proved a Poincar\'e inequality for Gaussian measures with respect to a different gradient operator}.

We stress that, to our purposes, it is key to have at one's disposal a \emph{vector-valued} Poincar\'e inequality. Indeed, even if $ F $ is scalar valued, its $\ell$-th Malliavin derivative becomes a vector-valued function in $ L^q(\Omega;\mH^{\otimes \ell}) $, hence we need to be able to apply \eqref{eq_poincare1} with $ \mV \equiv \mH^{\otimes \ell} $. Clearly, when $ \mV $ is finite dimensional, it is straightforward to pass from a scalar-valued to a vector-valued Poincar\'e inequality (up to multiplying constants), but if $ \mV $ is infinite dimensional it is not for granted.

Finally, in the spirit of Remark \ref{rem-comp}, let us observe that when both $ \mH $ and $ \mV $ are finite dimensional the Poincar\'e inequality, for any $q\in [1,\infty)$ and even measures more general than the Gaussian one, can be shown by means of a classical compactness argument, taking advantage of the embeddings proved in \cite{Hoo81}. However, in this case $ c_q $ is not quantitative.
  
 \begin{remark}[On the best Poincar\'e constant]\rm
 Finding the exact value of the \emph{optimal constant} $ c^{\text{opt}}_q $, for which the Poincar\'e inequality \eqref{eq_poincare1} holds, is a challenging open problem. Plainly, from (\ref{cost_poincare}) we have the upper bound 
 \begin{equation}
 \label{ub_const}
 c^{\text{opt}}_q  \le 
 \begin{cases}
 \sqrt{q-1} & \text{for } q\in [2,\infty) \, , \\
 \tfrac \pi 2  & \text{for } q\in[1,2) \, .
 \end{cases}
 \end{equation}
We aim at showing that such bound is at least asymtotically correct as $ q \to \infty $. To this purpose, let $ N := W(h) $, $ \| h \|_\mH = 1 $, be a standard Gaussian random variable. For any $q\in [1,\infty)$, we have:
 \begin{equation*}
 \left\| N - \mathbb{ E}[N] \right\|_{L^q(\Omega)}^q =  \mathbb E\!\left[|N|^q\right] = \frac{2^{\frac{q}{2}}}{\sqrt \pi} \, \Gamma\!\left (\tfrac{q+1}{2} \right ) \quad \text{and} \quad  \left\| D N \right\|_{L^q(\Omega;\mH)}^q = \mathbb{E}\!\left[ \left\| h \right\|_\mH^q  \right] = 1 \, ,
 \end{equation*}
 where $\Gamma$ denotes the Euler gamma function, thus
 \begin{equation}\label{lb_const}
      c_{q}^\text{opt} \ge \frac{\sqrt 2}{\pi^{ \frac {1}{2q} }} \, \Gamma\!\left (\tfrac{q+1}{2} \right )^{\frac{1}{q}} .
 \end{equation}
Note that from (\ref{ub_const}) and (\ref{lb_const}), for $q=2$, we rediscover the well-known result $c^\text{opt}_2=1$. Moreover,
 a simple computation shows that, at the level of orders of magnitude, the upper and lower bounds in \eqref{ub_const}--\eqref{lb_const} have the same asymptotic behavior as $q\to \infty$.
 Indeed, it suffices to notice that, due to Stirling's formula, the right-hand side of \eqref{lb_const} behaves like $\sqrt{q}/\sqrt{e}$ as $q\to \infty$.
\end{remark}
 
\subsubsection{Norms of (expected) Malliavin derivatives}

We now focus on the second addend on the right-hand side of \eqref{app1}. Here, and in the sequel, the symbol $ \overline q $ stands for the conjugate exponent of $ q $, namely $ \overline q := q/(q-1) $.
\begin{lemma}\label{prop2fd}
Let $(\mathcal H, \langle \cdot, \cdot \rangle_{\mathcal H})$ and $(\mathcal V, \langle \cdot, \cdot \rangle_{\mathcal V})$ be any real separable Hilbert spaces, $ \ell \in \N_{\ge 1} $ and $ q \in (1,\infty) $. Then, for all $F\in \mathcal S_{\mathcal V}$, we have
\begin{equation}\label{stima_exp}
    \left\| \mathbb E 	\! \left[ D^\ell F \right] \right\|_{\mH^{\otimes \ell} \otimes \mathcal V} \le  d_{\ell,q} \left\| F \right\|_{L^q(\Omega; \mathcal V)} ,
\end{equation}
 where
    \begin{equation}\label{const_exp}
   d_{\ell,q} := \begin{cases}
        \sqrt{\ell!}  & \text{for } q \in [2,\infty) \, ,\\
        \sqrt{\ell!} \left(\overline{q} -1\right)^{\frac \ell 2}  & \text{for }  q\in (1,2) \, .
        \end{cases}
    \end{equation}
\end{lemma}
An analogue of Lemma \ref{prop2fd} was proved in \cite[Corollary 3.2]{PV14}, but again without quantitative constants. Our proof is elementary and deeply relies on the well-known relation between expected Malliavin derivatives and \emph{projections on Wiener chaoses} (see \S \ref{sec_mean}). Note that, regardless of the fact that $ d_{\ell,q} \uparrow +\infty $ as $ q \downarrow 1 $, inequality \eqref{stima_exp} necessarily fails for $q=1$ (see also Remark \ref{equivalence-1}). Nevertheless, we are able obtain the following modified version of Lemma \ref{prop2fd}.
\begin{lemma}\label{prop2}
Let $(\mathcal H, \langle \cdot, \cdot \rangle_{\mathcal H})$ and $(\mathcal V, \langle \cdot, \cdot \rangle_{\mathcal V})$ be any real separable Hilbert spaces and $ \ell \in \N_{\ge 1} $.
Then, for all $F\in\mathcal S_{\mV}$, we have
\begin{equation} \label{d2}
\begin{gathered}
\left\|\mathbb E\!\left[D^\ell F \right] \right\|_{\mathcal H^{\otimes \ell}\otimes \mathcal V}
\leq  18\sqrt{2e} \left(\rho^{-1} \left\|D^{\ell-1} F \right\|_{L^1(\Omega;\mathcal H^{\otimes \ell-1} \otimes \mathcal V)} + \rho \left\| D^{\ell+1}F \right\|_{L^1(\Omega;\mathcal H^{\otimes \ell+1} \otimes \mathcal V)} \right) \\
\forall \rho \in (0,1) \, .
\end{gathered}
\end{equation}
\end{lemma}

In the special case $ \ell=1 $, inequality \eqref{d2}  ensures that one can still control the norm of the expected Malliavin derivative with the $ L^1 $ norm of $F$ \emph{plus} the $L^1$ norm of the second Malliavin derivative.

Theorems \ref{thm2}, \ref{th1} (and thus Corollary \ref{coro_principale}) hence follow from Theorem \ref{prop_poincare} and Lemmas \ref{prop2fd}, \ref{prop2}, respectively, bearing in mind \eqref{app1}. For the details we refer to the proofs carried out in \S \ref{sec_proofs}. 

If $ \mH $ is finite dimensional, in particular, we can improve estimate \eqref{d2} by replacing the norm of the expected $\ell $-th Malliavin derivative with its $ L^1 $ norm.
\begin{lemma}\label{lem1fd}
Let $(\mathcal H, \langle \cdot, \cdot \rangle_{\mathcal H})$ and $(\mathcal V, \langle \cdot, \cdot \rangle_{\mathcal V})$ be any real Hilbert spaces, with $ \mH $ finite-dimensional and $ \mV $ separable. Let $ \ell \in \N_{\ge 1} $, $ q \in [1,\infty) $  and $ n:= \mathrm{dim}(\mH)  $. Then, for all $F\in \mathcal S_{\mathcal V}$, we have
	\begin{equation}\label{dav1}
	\left\| D^\ell F \right\|_{L^q(\Omega;\mathcal H^{\otimes \ell}\otimes \mathcal V)} \le C_{\ell,n} \left( \rho^{-\ell} \, \|F\|_{L^q(\Omega;\mathcal V)} + \rho \left\| D^{\ell+1} F \right\|_{L^q(\Omega; \mathcal H^{\ell+1} \otimes \mathcal V)} \right) \qquad \forall \rho \in (0,1) \, ,
	\end{equation}
where
\begin{equation}\label{zeta}
	C_{\ell,n} := 2^{\sum_{m=1}^{\ell-1} \frac{\ell !}{\left( \ell - m \right)!}} \, C_{1,n}^{\sum_{m=1}^{\ell} \frac{\ell !}{\left( \ell - m \right)!}} \quad \forall \ell \in \N_{\ge 2} \, , \qquad C_{1,n} := 18\sqrt{e} \, n \, .
\end{equation}
\end{lemma}
Theorem \ref{thm3} easily follows from Lemma \ref{lem1fd}, see again \S \ref{sec_proofs} for the details. Due to \eqref{zeta}, it is apparent that constants blow up as $ n \to \infty $, thus it is not possible to extend it to the infinite-dimensional setting. In the case $ q=1 $, the problem of estimate \eqref{d2} is that on the left-hand side we only have the norm of the \emph{expected} $ \ell $-th Malliavin derivative. If, instead, we had its full $ L^1 $ norm, by suitably tuning the free parameter $ \rho $ it would not be difficult to infer an analogue of \eqref{dav1}, which is precisely what we do in the finite-dimensional setting (we refer to the proof of Lemma \ref{lem1fd} in \S \ref{sec_propmean}).


\subsection{Previous work}\label{sec_previous}

Let us sum up what was previously known about the main problem that served as a motivation to this paper, that is the equivalence of the Sobolev norms $\|\cdot \|_{\mathcal G(k,q)(\mV)}$ and $\| \cdot\|_{\mathcal D(k,q)(\mV)}$. The very first proof of such equivalence, for $ q \in (1,\infty) $, dates back to Meyer's paper \cite{Mey84}, and deals with the case $\mathcal V \equiv \mathbb R$. Subsequently, Sugita \cite{Sug85} was able to extend Meyer's proof to the case of a general real separable Hilbert space $\mathcal V$. The key ingredient to such an approach are the so-called \emph{Meyer's inequalities}, see e.g.~\cite[Theorem 1.5.1]{Nua95}, which allow one to compare the $ L^q $ norm of $ D F $ with the $ L^q $ norm of a transformed function, obtained by suitably modifying the expansion of $F$ in terms of Wiener chaoses (see Appendix \ref{appendixA}).
Note that an alternative, probabilistic proof of Meyer's inequalities can be found in \cite{Gun89}, while for a short analytic one we refer to \cite{Pis88}.

More recently, these results were extended in \cite{PV14} to functions with values in a special class of Banach spaces, known as $UMD$. The strategy is similar to ours (in the case $ q \in (1,\infty) $), as the starting point of the authors is still \eqref{app1}: the first term is treated via the Poincar\'e inequality, whereas for the second one the continuity of the chaotic projection operator in $L^q(\Omega)$, for any $q\in (1,\infty)$, is invoked (see again Appendix \ref{appendixA}). Nevertheless, differently from our argument, the proof of \cite[Proposition 3.1]{PV14} (i.e.~the Poincar\'e inequality) exploits an appropriate generalization of Meyer's inequalities \cite[Theorem 2.6]{PV14}.

In conclusion, on the one hand, the previous results discussed encompass ours in the case $ q \in (1,\infty) $, but with very little quantitative knowledge on the multiplying constants involved. Moreover, all of them relying on Meyer-type inequalities, they are not applicable to the critical case $q=1$. On the other hand, to the best of our knowledge, no result was available in the literature for $q=1$, except those one could infer from finite-dimensional compact embeddings (recall Remark \ref{rem-comp}).

\subsection{Plan of the paper}

In \S\ref{sec_proofs} we prove our main results (Theorem \ref{thm2}, Theorem \ref{th1} and Theorem \ref{thm3}) upon \emph{assuming} the validity of crucial auxiliary results such as Theorem \ref{prop_poincare}, Lemma \ref{prop2fd}, Lemma \ref{prop2} and Lemma \ref{lem1fd}. 
In \S \ref{sec_poincare} we establish Theorem \ref{prop_poincare}, that is the Poincar\'e inequality, in its full generality. We then prove Lemma \ref{prop2fd} in \S \ref{sec_mean} and Lemmas \ref{prop2}, \ref{lem1fd} in \S\ref{sec_propmean}.
Finally, some technicalities {regarding the Wiener chaos and the proof of further auxiliary results} are collected in Appendices \ref{appendixA} and {\ref{appendixB}}.

\section{Proofs of the main results}\label{sec_proofs}

We are now in position to prove Theorem \ref{thm2}, Theorem \ref{th1} and Theorem \ref{thm3}.

\smallskip
\begin{proof}[Proof of Theorem \ref{thm2} assuming Theorem \ref{prop_poincare} and Lemma \ref{prop2fd}] 
Bearing in mind inequality (\ref{ineq1}), it suffices to prove that for any $k \in \mathbb N_{\ge 2}$, $q\in [1,\infty)$ and all $ F \in \mathcal{S}_{\mV} $ we have
\begin{equation}\label{induc-kq}
    \| F\|_{\mathcal D(k,q)(\mathcal V)} \le \tau_{k,q} \, \|F\|_{\mathcal G(k,q)(\mathcal V)} \, ,
\end{equation}
where $\tau_{k,q}$ is defined as in \eqref{const}. We will establish \eqref{induc-kq} by induction, recalling that it trivially holds for $ q=1 $ with $ \tau_{1,q} = 1 $. Given $ k \ge 2 $, let us therefore assume that the inequality is satisfied for $ k-1 $. By the definition of the norm $ \| \cdot \|_{\mathcal{D}(k,q)(\mV)} $, it holds 
$$
\begin{aligned}
\| F \|_{\mathcal{D}(k,q)(\mV)} = & \left( \| F \|_{\mathcal{D}(k-1,q)(\mV)}^q + \left\| D^k F \right\|_{L^q(\Omega;\mH^{\otimes k}\otimes \mV)}^q \right)^{\frac 1 q} \\
 \leq & \, \| F \|_{\mathcal{D}(k-1,q)(\mV)} + \left\| D^k F \right\|_{L^q(\Omega;\mH^{\otimes k}\otimes \mV)}.
\end{aligned}
$$
On the one hand, the induction hypothesis yields
$$
\begin{aligned}
\| F \|_{\mathcal{D}(k-1,q)(\mV)}  \le & \, \tau_{k-1,q} \, \| F \|_{\mathcal{G}(k-1,q)(\mV)}  \\
= & \, \tau_{k-1,q} \, \| F \|_{L^q(\Omega;\mV)} + \tau_{k-1,q} \left\| D^{k-1} F \right\|_{L^q(\Omega;\mH^{\otimes {k-1}}\otimes \mV)} ;
\end{aligned}
$$
on the other hand, estimates \eqref{app1} and \eqref{stima_exp} (applied to $ \ell = k-1 $) plus the Poincar\'e inequality \eqref{eq_poincare1} (applied to $ F \equiv D^{k-1} F $ and $ \mV \equiv \mH^{\otimes k-1} \otimes \mV $) entail
$$
\left\| D^{k-1} F \right\|_{L^q(\Omega;\mH^{\otimes {k-1}}\otimes \mV)} \le d_{k-1,q} \, \| F \|_{L^q(\Omega;\mV)} + c_q \left\| D^k F \right\|_{L^q(\Omega;\mH^{\otimes {k}}\otimes \mV)} \, .
$$ 
As a result of the last two formulas, and noticing that $ \tau_{k-1,q} \ge 1 $, we end up with 
$$
\begin{aligned}
\| F \|_{\mathcal{D}(k,q)(\mV)} \le & \, \tau_{k-1,q} \left(1+d_{k-1,q}\right) \| F \|_{L^q(\Omega;\mV)} + \tau_{k-1,q} \left(1+ c_q \right) \left\| D^k F \right\|_{L^q(\Omega; \mH^{\otimes k} \otimes \mV)} \\
\le & \left( 1 + d_{k-1} \vee c_q \right) \tau_{k-1,q} \left\| F \right\|_{\mathcal{G}(k,q)} .
\end{aligned}
$$
If $ k=2 $ we thus infer \eqref{induc-kq} with $ \tau_{2,q} $ as in \eqref{const}, recalling that $ \tau_{1,q}=1 $. If $ k \ge 3 $, then by the induction hypothesis $ \tau_{k-1,q} $ complies with \eqref{const}, so that
$$
\tau_{k,q} = \left( 1 + d_{k-1} \vee c_q \right) \tau_{k-1,q} = \left( 1 + d_{k-1} \vee c_q \right) \prod_{\ell=1}^{k-2} \left( 1+ d_{\ell,q} \vee c_q \right) = \prod_{\ell=1}^{k-1} \left( 1+ d_{\ell,q} \vee c_q \right) ,
$$
i.e., also $ \tau_{k,q} $ complies with \eqref{const}, and the proof is complete.
\end{proof}

\smallskip
\begin{proof}[Proof of Theorem \ref{th1} and Corollary  \ref{coro_principale} assuming Theorem \ref{prop_poincare} and Lemma \ref{prop2}]
Let $ \ell \in \mathbb{N}_{\ge 1} $. From \eqref{app1}, \eqref{d2} and the Poincar\'e inequality \eqref{eq_poincare1} (applied to $q=1$, $ F \equiv D^{\ell} F $ and $ \mV \equiv \mH^{\otimes \ell} \otimes \mV $), for all $ \rho \in (0,1) $ we obtain:
$$
\begin{aligned}
\left\| D^\ell F \right\|_{L^1(\Omega;\mathcal H^{\otimes \ell}\otimes \mathcal V)} \leq & \left\| \mathbb E\!\left[D^\ell F \right] \right\|_{\mathcal H^{\otimes \ell}\otimes \mathcal V} + \frac\pi2 \left\|D^{\ell+1}F\right\|_{L^1(\Omega;\mathcal H^{\otimes \ell+1}\otimes \mathcal V)} \\
\leq & \, 18\sqrt{2e} \, \rho^{-1} \left\| D^{\ell-1}F \right\|_{L^1(\Omega;\mathcal H^{\otimes \ell-1}\otimes \mathcal V)} \\
& \left(\frac\pi2+18\sqrt{2e} \, \rho \right) \left\|D^{\ell+1}F \right\|_{L^1(\Omega;\mathcal H^{\otimes \ell+1}\otimes \mathcal V)} ,
\end{aligned}
$$
so that by letting $ \rho \uparrow 1 $ we easily infer \eqref{stima_principale_1}. Corollary \ref{coro_principale}$(i)$ is a direct consequence of the definition of $ \| \cdot \|_{\mathcal{D}(2,1)(\mV)} $ and \eqref{stima_principale_1} with $ \ell = 1 $. As for $ (ii) $, by applying \eqref{stima_principale_1} to all odd derivatives (except the first and the $ k $-th), we have:
$$
\begin{aligned}
 & \, \|F\|_{\mathcal D(k,1)(\mathcal V)} \\
= & \, \|F\|_{\mathcal G(k,1)(\mathcal V)} \! + \sum_{\ell=1}^{ \lfloor  k/2 \rfloor} \left\| D^{2\ell-1}F \right\|_{L^1(\Omega;\mathcal H^{\otimes 2\ell-1}\otimes \mV)} + \sum_{\ell=1}^{ \lceil k/2 \rceil -1 } \left\| D^{2\ell} F \right\|_{L^1(\Omega;\mathcal H^{\otimes 2\ell}\otimes \mV)} \, \\
 \le & \, \|F\|_{\mathcal G(k,1)(\mathcal V)} + \eta \sum_{\ell=1}^{ \lfloor  k/2 \rfloor} \left\| D^{2\ell-2}F \right\|_{L^1(\Omega;\mathcal H^{\otimes 2\ell-2}\otimes \mV)} + \eta \sum_{\ell=1}^{ \lfloor  k/2 \rfloor} \left\| D^{2\ell}F \right\|_{L^1(\Omega;\mathcal H^{\otimes 2\ell}\otimes \mV)} \\
 & + \sum_{\ell=1}^{ \lceil k/2 \rceil -1 } \left\| D^{2\ell} F \right\|_{L^1(\Omega;\mathcal H^{\otimes 2\ell}\otimes \mV)} 
 \\
 \le & \, \|F\|_{\mathcal G(k,1)(\mathcal V)} + \eta \, \| F \|_{L^1(\Omega;\mV)} + \eta \left\| D^k F \right\|_{L^1(\Omega;\mH^{\otimes k} \otimes \mV)} \\
& + (2\eta+1) \sum_{\ell=1}^{ \lceil k/2 \rceil -1 } \left\| D^{2\ell} F \right\|_{L^1(\Omega;\mathcal H^{\otimes 2\ell}\otimes \mV)} ,  \\
\end{aligned}
$$ 
which yields \eqref{aa}. The proof of \eqref{bb} is completely analogous.
\end{proof}

\smallskip
\begin{proof}[Proof of Theorem \ref{thm3} assuming Lemma \ref{lem1fd}]
Again, we only need to prove the rightmost inequality in \eqref{estimate_fdc3}. Given $ \varepsilon \in (0,1) $, let us pick 
$$ 
\rho = \frac{\varepsilon}{2 \, C_{\ell , n}}
$$
in \eqref{dav1}, for all $ \ell = 1, \ldots, k-1 $. Note that such a choice is feasible since $ C_{\ell,n}>1 $. As a result, we obtain: 
$$
\begin{gathered}
	\left\| D^\ell F \right\|_{L^q(\Omega;\mathcal H^{\otimes \ell}\otimes \mathcal V)} \le \frac{2^\ell \, C_{\ell,n}^{\ell+1}}{\varepsilon^\ell} \, \|F\|_{L^q(\Omega;\mathcal V)} + \frac \varepsilon 2 \left\| D^{\ell+1} F \right\|_{L^q(\Omega; \mathcal H^{\ell+1} \otimes \mathcal V)} \\ \forall \ell = 1 , \ldots , k-1 \, .
	\end{gathered}
$$
Adding up both sides of the inequalities, and observing that the constants $ C_{\ell,n} $ are increasing w.r.t.~$ \ell $, we end up with 
$$
\begin{aligned}
\sum_{\ell=1}^{k-1} \left\| D^\ell F \right\|_{L^q(\Omega;\mathcal H^{\otimes \ell}\otimes \mathcal V)} 
\le & \, \frac \varepsilon 2 \, \sum_{\ell=1}^{k-1} \left\| D^\ell F \right\|_{L^q(\Omega;\mathcal H^{\otimes \ell}\otimes \mathcal V)} + \frac \varepsilon 2 \left\| D^k F \right\|_{L^q(\Omega;\mathcal H^{\otimes k}\otimes \mathcal V)} \\
& + \sum_{\ell = 1}^{k-1} \frac{2^\ell}{\varepsilon^\ell} \,  C_{k-1,n}^{k} \, \|F\|_{L^q(\Omega;\mathcal V)} \, ,
\end{aligned}
$$
whence (recalling that $ \varepsilon <1 $)
$$
\sum_{\ell=1}^{k-1} \left\| D^\ell F \right\|_{L^q(\Omega;\mathcal H^{\otimes \ell}\otimes \mathcal V)} \le  \left( \frac{2^{k+1}}{\varepsilon^{k-1}} - 4 \right) C_{k-1,n}^{k} \, \|F\|_{L^q(\Omega;\mathcal V)}  + \varepsilon \left\| D^k F \right\|_{L^q(\Omega;\mathcal H^{\otimes k}\otimes \mathcal V)} ,
$$
from which \eqref{estimate_fdc3} readily follows since
$$
\left(\sum_{\ell=1}^{k} \left\| D^\ell F \right\|^q_{L^q(\Omega;\mathcal H^{\otimes \ell}\otimes \mathcal V)}\right)^{\frac{1}{q}}  \le \sum_{\ell=1}^{k} \left\| D^\ell F \right\|_{L^q(\Omega;\mathcal H^{\otimes \ell}\otimes \mathcal V)}  \, .
$$
\end{proof}

\begin{remark}[An equivalent analytic reformulation]\label{equiv-ref}\rm
Given the density of $ \mathcal{S}_\mV $ in $ \mathbb{D}^{k,q}(\mV) $, it is plain that our results can also be reformulated by using a pure analytical language. For instance, the rightmost inequality in \eqref{estimate_fdc} is equivalent to the validity of the inequality (for any $ n,m \in \N_{\ge 1}$)
\begin{equation}\label{finite-dim-example}
\begin{gathered}
\left\| f \right\|_{W^{k,q}(\R^n, \gamma_n;\R^{m})} \le \tau_{k,q} \left( \left\| f \right\|_{L^q(\R^n,\gamma_n; \R^m)} + \left\| \nabla^k f \right\|_{L^q \left(\R^n,\gamma_n; \R^{m^k}\right)} \right) \\ \forall f \in W^{k,q}(\R^n, \gamma_n;\R^{m}) \, ,
\end{gathered}
\end{equation} 
where $ W^{k,q}(\R^n, \gamma_n;\R^{m}) $ stands for the Sobolev space of locally integrable functions $ f: \R^n \to \R^m $ such that every component of $ f $ is weakly differentiable $ k $ times, and each of the corresponding partial derivatives is $q$-integrable with respect to the Gaussian measure
$$
\gamma_n(dx) := \left(2\pi\right)^{-\frac n 2} e^{-\frac{|x|^2}{2}} \, dx \, .
$$
In general, in order to be able to deal with infinite-dimensional Hilbert spaces, a possible strategy consists in proving finite-dimensional inequalities like \eqref{finite-dim-example}, whose constants are dimension independent. Especially in Section \ref{sec_propmean}, we will take advantage of such an equivalence.
\end{remark}

\section{Poincar\'e inequality: proof of Theorem \ref{prop_poincare}}\label{sec_poincare}

Here, and in the subsequent sections, we adopt the same notations as in \S \ref{sec_back}. In agreement with the definition of $ \mathcal{S}_\mV $, if $ \mathcal{T} $ is any subspace of $ L^1(\Omega) $, we let $ \mathcal{T}_\mV $ denote the space of $ \mV $-valued functions of the form
$$
F = \sum_{j=1}^J F_j \, v_j \, ,
$$
with $J\in \mathbb N_{\ge 1}$, $F_j\in \mathcal T$ and $v_j\in \mV$ for every $ j=1 , \ldots , J $.

First of all, let us introduce the Ornstein-Uhlenbeck semigroup defined by Mehler's formula:
\begin{align}\label{def_OU}
P_t F := \mathbb E' \!\left[f\!\left(e^{-t} \, W+\sqrt{1-e^{-2t}} \,W' \right) \right] \qquad \forall F\in L^1(\Omega) \, , \ \forall t \ge 0 \, , 
\end{align}
where $W'$ is an independent copy of $W$ defined on $(\Omega',\mathcal F',\mathbb P')$ and $\mathbb E'$ is the expectation with respect to $ \mathbb{P}'$. 
It is well known that $(P_t)_{t\geq0}$ extends to a positive strongly continuous semigroup of contractions on $L^q(\Omega)$, for any $q\in[1,\infty)$, and we denote by $ \mathcal{L}_q$ its infinitesimal generator in $L^q(\Omega)$.  For an in-depth analysis of the Ornstein-Uhlenbeck semigroup, we refer to \cite[Section 1.4]{Nua95}.

\begin{remark}[Self-adjointness]\label{rmk:forma_dir_L_2} \rm
We recall that $ \mathcal{L}_2$ turns out to be the linear self-adjoint operator associated to the bilinear form
$$
\mathcal E(F,G):=\mathbb E\!\left[\langle DF,DG\rangle_{\mathcal H}\right] \qquad \forall F,G\in \mathsf{D} (\mathcal E) := \mathbb D^{1,2} \, ,
$$
where $ \mathsf{D}(\cdot) $ stands for the domain of a form or an operator.
\end{remark}

If $ (\mV,\langle \cdot , \cdot \rangle_\mV)$ is a real separable Hilbert space, then from \cite[\S 2]{ta08} it follows that $(P_t)_{t\geq0}$ uniquely extends to a strongly continuous semigroup of contractions on $L^q(\Omega;\mV)$, which we denote by $(P^\mV_t)_{t\geq0}$, in the sense that if $F\in \mathcal S_{\mV}$ is of the form $F=\sum_{j=1}^JF_j \, v_j$, then
\begin{align*}
P^{\mV}_tF=\sum_{j=1}^JP_tF_j \, v_j  \qquad \forall t\geq0 \, .
\end{align*}
Accordingly, if we let $\mathcal{L}^\mV_q$ be the infinitesimal generator of $(P^\mV_t)_{t\geq0}$ in $L^q(\Omega;\mV)$, it is apparent that $\mathsf{D}_\mV(\mathcal{L}_q) \subset \mathsf{D}(\mathcal{L}_q^\mV)$ and $ (\mathcal{L}_q)_\mV  = \mathcal{L}_q^\mV$ (with some abuse of notation) on $\mathsf{D}_\mV(\mathcal{L}_q)$ for any $q\in[1,\infty)$, i.e., also the generator acts componentwise on the elements of $\mathsf{D}_\mV(\mathcal{L}_q)$.

Our next goal, in order to prove the $ \mV $-valued Poincar\'e inequality \eqref{eq_poincare1}, is to investigate further the basic properties of $ (P^\mV_t)_{t \ge 0} $ relying on what is known for the scalar semigroup $ (P_t)_{t \ge 0} $.
Having in mind Remark \ref{rmk:forma_dir_L_2}, we will now prove that also the semigroup $(P_t^\mV)_{t\geq0}$ in $L^2(\Omega;\mV)$ can be defined by means of the theory of Dirichlet forms. Let
$$
\mathcal E^\mV(F,G):=\mathbb E\!\left[\langle DF , DG\rangle_{\mH \otimes \mV}\right] \qquad \forall F,G\in \dom\!\left(\mathcal E^\mV\right):=\mathbb D^{1,2}(\mV) \, .    
$$
It is not difficult to check that $ \mathcal{E}^\mV $ is a positive symmetric closed bilinear form which satisfies the estimate $|\mathcal E^\mV(F,G)|\leq \mathcal E^\mV(F,F)^{1/2}\mathcal E^\mV(G,G)^{1/2}$, hence from \cite[Theorem 2.8]{MaRo92} there exists a strongly continuous semigroup of contractions $(T^\mV_t)_{t\geq0}$ on $L^2(\Omega;\mV)$ such that its infinitesimal generator $(\tilde \lop^\mV, \dom(\tilde \lop^\mV))$ is given by (recall the density of $ \mathcal{S}_\mV $ in $ \mathbb{D}^{1,2}(\mV) $)
\begin{align*}
\dom \big(\tilde \lop^\mV\big) := & \left\{F\in \mathbb D^{1,2}(\mV): \ \exists \Phi_F\in L^2(\Omega;\mV) \ \text{s.t.} \ \mathcal E^\mV(F,G)=-\mathbb E\!\left[\langle \Phi_F,G\rangle_\mV \right] \ \forall G \in \mathcal{S}_\mV \right\} , \\
\tilde \lop^\mV F := & \, \Phi_F \, ,
\end{align*}
which turns out to be a linear, self-adjoint and dissipative operator. 

We are going to prove that in fact $ (\lop^\mV_2,\dom(\lop^\mV_2)) =  (\tilde \lop^\mV,\dom(\tilde \lop^\mV))  $, hence $ (P^\mV_t)_{t \ge 0} $ and $ (T^\mV_t)_{t \ge 0} $ actually coincide. To this aim, we need the two following intermediate results, whose proofs are postponed to Appendix \ref{appendixB} for the reader's convenience.

\begin{lemma}
\label{lem:core_L_2}
For any $q\in[1,\infty)$, the space $\mathsf{D}_\mV(\mathcal{L}_q)$ is a core for $\mathsf{D}(\mathcal{L}_q^\mV)$.
\end{lemma}

\begin{lemma}
\label{lemma:core_tilde_L}
The space $\dom_\mV(\lop_2) $ is a core for $ \dom (\tilde \lop^\mV)$.
\end{lemma}

\begin{proposition}
\label{prop:car_OU}
We have $(\lop_2^\mV,\dom(\lop_2^\mV))=(\tilde \lop^\mV,\dom(\tilde \lop^\mV))$.
\end{proposition}
\begin{proof}
The proof of Lemma \ref{lemma:core_tilde_L} shows that $ \lop _2^\mV=\tilde \lop^\mV$ on $ \dom(\lop_2^\mV) $, which is a core for both operators, and two closed operators that share the same core must coincide.
\end{proof}

\smallskip 
The following preliminary results are well known for scalar-valued functions, and we state here their extension to $\mV$-valued functions; we provide complete proofs for the reader's convenience since the techniques involve tools related to those we will use in the proof of Theorem \ref{prop_poincare}.
\begin{proposition}
\label{prop:stima_grad_grad_scalare}
For any $q\in[1,\infty)$, all $F\in \mathbb D^{1,q}$ and all $t>0$ we have
\begin{align}\label{eq00}
\left\| D \!\left( P_t F \right) \right\|_{\mathcal H}^q \leq e^{-qt} P_t \left( \left\|D F \right\|_{\mathcal H}^q \right) \qquad \mathbb{P}\text{-a.s.~in }\Omega \, .
\end{align}
\end{proposition}
\begin{proof}
For $F \in \mathcal{S} $ the thesis can be shown by explicit computations taking advantage of \eqref{def_OU}, and for a general $F\in \mathbb D^{1,q}$ it follows from an approximation argument, using the density of $  \mathcal{S} $ in $ \mathbb D^{1,q}$ along with the fact that if $ F_n \to F $ in $ \mathbb D^{1,q}$ then both sides of \eqref{eq00} converge in $ L^1(\Omega) $.
\end{proof}

\begin{proposition}\label{p1}
For any $q\in(1,\infty)$, all $F\in L^q(\Omega)$ and all $ t>0 $ we have
\begin{align*}
\left\| D\!\left(P_t F\right) \right\|_{\mathcal H}^q \leq \left( \frac{e^{-t}}{\sqrt{1-e^{-2t}}} \right)^q \left(\int_\R |x|^{\overline q} \, \gamma_1(dx)\right)^{{q}/{\overline q}} P_t \left(|F|^q\right) \qquad \mathbb{P}\text{-a.s.~in } \Omega \, .
\end{align*}
\end{proposition}
\begin{proof}
First of all, let $F\in \mathcal{S}$ with $F=f(W)$. Explicit computations involving a routine change of variable under $ \E' $ give
\begin{align*}
\left\langle D \!\left( P_t F \right) \! , h \right\rangle_{\mathcal H} =
\frac{e^{-t}}{\sqrt{1-e^{-2t}}} \, \mathbb E' \! \left[W'(h)\, f\!\left(e^{-t} \, W+\sqrt{1-e^{-2t}} \, W'\right)\right] ,
\end{align*}
where $ h $ is an arbitrary element of $ \mH $. Moreover, by the definition of $ \|\cdot \|_\mH $ we have
\begin{align*}
\left\| D\!\left( P_t F \right) \right\|_{\mathcal H}=\sup_{h\in \mathcal H: \, \|h\|_{\mathcal H}=1}  \left\langle D\!\left( P_t F \right) \! , h \right\rangle_{\mathcal H}.
\end{align*}
Let us estimate $|\langle D\!\left( P_t F \right) \! , h \rangle_{\mathcal H}|$ for any $h\in \mathcal H$ with $\|h\|_{\mathcal H}=1$. By H\"older's inequality, we obtain:
\begin{align*}
|\langle D\!\left( P_t F \right) \! , h \rangle_{\mathcal H}|
= & \, \frac{e^{-t}}{\sqrt{1-e^{-2t}}}\left| \mathbb E'\! \left[W'(h)\, f\!\left(e^{-t} \, W+\sqrt{1-e^{-2t}} \, W'\right)\right] \right| \\
\leq & \, \frac{e^{-t}}{\sqrt{1-e^{-2t}}}\left(\int_\R|x|^{\overline q} \, \gamma_1(d x)\right)^{1/\overline q}\left[P_t\left(|F|^q\right)\right]^{1/q} \qquad  \mathbb{P}\text{-a.s.~in } \Omega \, ,
\end{align*}
whence, taking the supremum over $h$,
\begin{align*}
\left\| D\!\left( P_t F \right) \right\|_{\mathcal H} \leq \frac{e^{-t}}{\sqrt{1-e^{-2t}}}\left(\int_\R|x|^{\overline q} \, \gamma_1(d x)\right)^{1/\overline q}\left[P_t\left(|F|^q\right)\right]^{1/q} \qquad  \mathbb{P}\text{-a.s.~in } \Omega \, .
\end{align*}
The inequality for a general $ F \in L^q(\Omega) $ follows by means of an approximation argument similar to the one outlined in the proof of Proposition \ref{prop:stima_grad_grad_scalare}.
\end{proof}

\smallskip
We aim at extending the above estimates to the $ \mV $-valued semigroup $  (P_t^\mV)_{t \ge 0} $, at least for $ q \ge 2 $.
 
\begin{proposition}
\label{prop:stima_derivata_funzione_vett}
For any $q\in[2,\infty)$, all $ F \in L^q(\Omega;\mV) $ and all $ t>0 $ we have
\begin{align}\label{prop:stima_derivata_funzione_vett_01}
\left\| D\!\left(P_t^\mV F\right) \right\|_{\mathcal H \otimes \mV}^q \leq \left(\frac{e^{-t}}{\sqrt{1-e^{-2t}}}\right)^q \, P_t \left( \left\| F \right\|_{\mV}^q \right) \qquad  \P\text{-a.s.~in } \Omega
\end{align}
and
\begin{align*}
\mathbb E\! \left[ \left\| D\!\left(P_t^\mV F\right) \right\|_{\mathcal H \otimes \mV}^q \right]
\leq \left(\frac{e^{-t}}{\sqrt{1-e^{-2t}}}\right)^q \, \mathbb E \left[ \left\| F \right\|_{\mV}^q \right].
\end{align*}
\end{proposition}
\begin{proof}
The second part of the statement follows directly from the first one and the fact that $\mathbb E[P_t G]=\mathbb E[G]$ for all $G\in L^1(\Omega)$. Hence, we only have to prove the first part of the statement. To this end, let $F  \in \mathcal{S}_\mV $ be of the form
$$
F = \sum_{j=1}^J F_j \, v_j \, ,
$$
with $J\in \mathbb N_{\ge 1}$, $F_j \in \mathcal{S}$ and $v_j\in \mV$ for every $ j=1 , \ldots , J $. 
Without loss of generality we can assume that $ \{ v_j \}_{j=1,\ldots,J} $ are orthonormal vectors in $\mV$. From Proposition \ref{p1} in the case $q=2$, and the linearity of $ (P_t)_{t \ge 0} $, we infer that
\begin{align*}
\left\| D \!\left( P_t^\mV F \right) \right\|_{\mathcal H\otimes \mV}^2
= & \sum_{j=1}^J \left\|D\!\left( P_t F_j \right) \right\|_{\mathcal H}^2
\leq \left(\frac{e^{-t}}{\sqrt{1-e^{-2t}}}\right)^2\left(\int_\R |x|^{2} \, \gamma_1(d x)\right) \sum_{j=1}^J P_t\!\left(|F_j|^2\right) \\
= & \left(\frac{e^{-t}}{\sqrt{1-e^{-2t}}}\right)^2 \, P_t \! \left(\sum_{j=1}^J |F_j|^2\right) \\
= & \left(\frac{e^{-t}}{\sqrt{1-e^{-2t}}}\right)^2 \,  P_t\!\left( \|F\|^2_\mV \right) \qquad \P \text{-a.s.~in } \Omega \, ,
\end{align*}
namely \eqref{prop:stima_derivata_funzione_vett_01} for $ q=2 $. On the other hand, for $ q>2 $ it holds (still $ \P \text{-a.s.~in } \Omega $)
\begin{align*}
\left\| D \!\left( P_t^\mV F \right) \right\|_{\mathcal H\otimes \mV}^q
\leq \left(\frac{e^{-t}}{\sqrt{1-e^{-2t}}}\right)^q \, \left[ P_t\!\left( \|F\|^2_\mV \right) \right]^{\frac q 2} 
\leq \left(\frac{e^{-t}}{\sqrt{1-e^{-2t}}}\right)^q \, P_t \left( \left\| F \right\|_{\mV}^q \right) ,
\end{align*}
where in the last passage we have applied H\"older's inequality to the integral representation of $P_t$. Again, the case of a general $ F \in L^q(\Omega;\mV) $ follows by approximation.
\end{proof}

\smallskip
Arguing as in Proposition \ref{prop:stima_derivata_funzione_vett} and taking advantage of Proposition \ref{prop:stima_grad_grad_scalare} instead, we obtain the next result, whose proof is omitted.

\begin{proposition}
\label{prop:stima_derivata_derivata_vett}
For any $q\in[2,\infty)$, all $ F \in \mathbb{D}^{1,q}(\mV) $ and all $ t>0 $ we have
\begin{align*}
\left\| D \! \left( P^\mV_t F \right) \right\|_{\mathcal H \otimes \mV}^q
\leq e^{-qt} \, P_t \left( \left\| D F \right\|^q_{\mathcal H \otimes \mV} \right) \qquad  \P\text{-a.s.~in } \Omega
\end{align*}
and
\begin{align*}
\E \!\left[ \left\| D \! \left( P^\mV_t F \right) \right\|_{\mathcal H \otimes \mV}^q \right]
\leq e^{-qt} \, \E \!\left[ \left\| D F \right\|^q_{\mathcal H\otimes \mV}\right].
\end{align*}
\end{proposition}

The last preliminary result we need concerns the asymptotic behavior of $(P^\mV_t)_{t\geq0}$ in $L^q(\Omega;\mV)$ as $ t \to +\infty $. 

\begin{lemma}
\label{lemma:asymp_behaviour}
For any $q\in[1,\infty)$ and all $ F \in L^q(\Omega;\mV)$ we have
\begin{align*}
\lim_{t\rightarrow+\infty} P^\mV_t F =\mathbb E[F] \qquad \text{in } L^q(\Omega;\mV) \, .
\end{align*}
\end{lemma}
\begin{proof}
By the density of $ \mathcal{S}_\mV $ in $ L^q(\Omega;\mV) $ and the fact that $ (P_t^\mV)_{t \ge 0} $ is a semigroup of contractions on $ L^q(\Omega;\mV) $, we can limit ourselves to proving the statement for all $F  \in \mathcal{S}_\mV $ of the form
$$
F = \sum_{j=1}^J F_j \, v_j \, ,
$$
with $J\in \mathbb N_{\ge 1}$, $F_j \in \mathcal{S}$ and $v_j\in \mV$ for every $ j=1 , \ldots , J $. As above, we may assume that $\{v_j\}_{j=1,\ldots,J}$ are orthonormal vectors in $\mV$. Then, given any $ t>0 $, it holds
\begin{align*}
P_t^\mV F  
= \sum_{j=1}^J P_t F_j \, v_j \quad \implies
\quad     
\left\| P_t^\mV F \right\|_\mV^2= \sum_{j=1}^J \left| P_t F_j \right|^2 
\end{align*}
and
\begin{align*}
\mathbb E\!\left[\left\| P^\mV_t F - \mathbb E[F] \right\|^q_\mV \right]
=\mathbb E \!\left[\left(\sum_{j=1}^J \left| P_t F_j - \mathbb E[F_j] \right|^2\right)^{\frac q 2}\right]
 \rightarrow0 \qquad \text{as } t\rightarrow+\infty \, ,
\end{align*}
since $P_tf\rightarrow\mathbb E[f]$ in $L^q(\Omega)$ as $t\rightarrow+\infty$ for any $f\in L^q(\Omega)$.
\end{proof}

\smallskip
We are now able to prove the claimed vector-valued Poincar\'e inequality.

\smallskip
\begin{proof}[Proof of Theorem \ref{prop_poincare}]
Let us split the proof into two steps. In the former we establish the statement for $q \in[2,\infty)$, in the latter we conclude with the case $ q \in [1,2) $ through a duality argument. 

\smallskip
{\bf Step $\boldsymbol 1$.}
Let $q\in[2, \infty)$, $F \in \mathcal{S}_\mV $ and set $G:=F-\mathbb E[F]$. By virtue of a standard approximation procedure, it is not hard to check that the function $G^*:=\left\|G \right\|_\mV^{q-2}G$ belongs to $\mathbb D^{1,p}(\mV)$ for any $ p \in [1,\infty) $, and
\begin{align}
\label{der_funz_duale}
DG^*=(q-2)\|G\|_\mV^{q-4} \, \langle G,DF\rangle_\mV \otimes G 
+ \|G\|^{q-2}_\mV \, DF \, ,
\end{align} 
where we mean
\begin{align*}
\langle G,DF\rangle_{\mV}:=\sum_{j=1}^J[F_j-\mathbb E[F_j]]DF_j \, , \quad F=\sum_{j=1}^JF_j \, v_j \, ,
\end{align*}
the components $F_1,\ldots,F_J$ belonging to $ \mathcal S$ and $ \{ v_j \}_{j=1, \ldots,  J}  $ being orthonormal vectors of $\mV$. 
From Lemma \ref{lemma:asymp_behaviour}, we have:
\begin{align}
\notag
\mathbb E\!\left[\left\|G\right\|^q_\mV\right]
= \mathbb E\!\left[\langle F-\mathbb E[F],G^*\rangle_\mV\right]
= & \lim_{s\rightarrow+\infty}\mathbb E[\langle P_0^{\mV}F-P^{\mV}_sF,G^*\rangle_{\mV} \\
= &  -\lim_{s\rightarrow+\infty}\mathbb E \! \left[\left\langle \int_0^s \frac{d}{dt}\left(P^\mV_t F \right) dt , G^* \right\rangle_{\!\mV} \right] \notag \\
= & - \lim_{s\rightarrow+\infty}\mathbb E \! \left[ \int_0^s \left\langle \lop_2^\mV \!\left( P^\mV_t F \right)\!,G^*\right\rangle_{\mV} dt \right] .
\label{eq-gstar}
\end{align}
On the other hand, Proposition \ref{prop:car_OU} and the definition of $\tilde \lop^\mV$ give, for all $s>0$,
\begin{align*}
-\mathbb E \! \left[ \int_0^{s} \left\langle \lop_2^\mV \!\left( P^\mV_t F \right)\!,G^*\right\rangle_{\mV} dt \right]
= & - \int_0^{s} \mathbb E \! \left[ \left\langle \tilde\lop^\mV \!\left( P^\mV_t F \right)\!,G^*\right\rangle_{\!\mV} \right] dt \\ 
= & \int_0^{s} \mathbb E \! \left[\left \langle D \!\left( P^V_t F \right) \! , DG^*\right\rangle_{\mathcal H \otimes \mV}\right]dt \, .
\end{align*}
By applying \eqref{der_funz_duale} we get
\begin{align*}
& \left|  \int_0^{s}  \mathbb E \! \left[ \left\langle D \! \left( P^\mV_t F \right)\!,DG^*\right\rangle_{\mathcal H\otimes \mV} \right] dt \right| \\
\leq & \, (q-1) \int_0^{s} \mathbb E \! \left[ \left\|F-\mathbb E[F]\right\|_\mV^{q-2} \left\|DF\right\|_{\mathcal H\otimes \mV} \left\| D\!\left(P^\mV_tF\right) \right\|_{\mathcal H\otimes \mV}\right]dt \, .
\end{align*}
Let us consider two different cases. If $q=2$, we obtain
\begin{align*}
\left|  \int_0^{s}  \mathbb E \! \left[ \left\langle D \! \left( P^\mV_t F \right)\!,DG^*\right\rangle_{\mathcal H\otimes \mV} \right] dt \right|
\leq \int_0^{s} \mathbb E \! \left[ \left\|DF\right\|_{\mathcal H\otimes \mV} \left\| D\!\left(P^\mV_tF\right) \right\|_{\mathcal H\otimes \mV}\right]dt \, ,
\end{align*}
so that Proposition \ref{prop:stima_derivata_derivata_vett} yields
\begin{equation*}
\int_0^{s} \mathbb E \! \left[ \left\|DF\right\|_{\mathcal H\otimes \mV} \left\| D\!\left(P^\mV_tF\right) \right\|_{\mathcal H\otimes \mV}\right]dt
\leq  \mathbb E\! \left[\left\| DF \right\|_{\mathcal H\otimes \mV}^2\right]\int_0^{{s}} e^{-t} \, dt \leq \mathbb E\!\left[\left\| DF \right\|_{\mathcal H\otimes \mV}^2\right],
\end{equation*}
which implies the thesis upon letting $s \to +\infty$. Let us turn to the case $q>2$. If we apply H\"older's inequality with exponents $q$, $q$ and $ {q}/{(q-2)}$, still Proposition \ref{prop:stima_derivata_derivata_vett} entails (for all $t>0$)
\begin{align*}
&  \mathbb E \! \left[ \left\|F-\mathbb E[F]\right\|_\mV^{q-2} \left\|DF\right\|_{\mathcal H\otimes \mV} \left\| D\!\left(P^\mV_tF\right) \right\|_{\mathcal H\otimes \mV}\right] \\
\leq & \left(\mathbb E\!\left[\left\|F-\mathbb E[F]\right\|_\mV^q\right]\right)^{1-\frac 2 q} \left( \mathbb E \! \left[\left\|DF\right\|_{\mH\otimes \mV}^q\right] \right)^{\frac 1 q} \left(\mathbb E\!\left[\left\| D \!\left( P^\mV_t F \right) \right\|_{\mH \otimes \mV}^q\right]\right)^{\frac 1 q} \\
\leq & \, e^{-t}  \left(\mathbb E\!\left[\left\|F-\mathbb E[F]\right\|_\mV^q\right]\right)^{1-\frac 2 q} \left( \mathbb E \! \left[\left\|DF\right\|_{\mH\otimes \mV}^q\right] \right)^{\frac 2 q} .
\end{align*}
Returning to \eqref{eq-gstar} and integrating in time, this yields
\begin{equation*}
\mathbb E \! \left[\left\|F-\mathbb E[F]\right\|^q_\mV\right]
\leq  (q-1) \left(\mathbb E\!\left[\left\|F-\mathbb E[F]\right\|_\mV^q\right]\right)^{1-\frac 2 q} \left( \mathbb E \! \left[\left\|DF\right\|_{\mH\otimes \mV}^q\right] \right)^{\frac 2 q} ,
\end{equation*}
and multiplying both sides by $\left(\mathbb E\!\left[\left\|F-\mathbb E[F]\right\|_\mV^q\right]\right)^{ 2 /q-1}$ (we can assume with no loss of generality that such quantity is not zero) we end up with
\begin{align*}
\left(\mathbb E \! \left[\left\|F-\mathbb E[F]\right\|^q_\mV\right]\right)^{\frac 2 q}
\leq (q-1) \left( \mathbb E \! \left[\left\|DF\right\|_{\mH\otimes \mV}^q\right] \right)^{\frac 2 q} . 
\end{align*}

\smallskip
{\bf Step $\boldsymbol 2$.}
Let $q\in(1,2)$ and $F \in \mathcal{S}_{\mV} $. By the density of $ \mathcal{S}_\mV $ in $ L^{\overline{q}}(\Omega;\mV) $ and Lemma \ref{lemma:asymp_behaviour}, we have: 
\begin{align*}
\left(\mathbb E \! \left[ \left\|F-\mathbb E[F] \right\|_\mV^q\right]\right)^{\frac 1 q}
=& \sup_{G \in \mathcal{S}_\mV : \, \| G \|_{L^{\overline{q}}(\Omega;V)}\leq 1}\mathbb E \! \left[\langle F-\mathbb E[F],G\rangle_\mV\right]\\
= & \sup_{G \in \mathcal{S}_\mV : \, \| G \|_{L^{\overline{q}}(\Omega;V)}\leq 1} \, \lim_{t \to +\infty} \mathbb E\!\left[\left\langle F-P_t^\mV F,G\right\rangle_\mV \right] \\
= & \sup_{G \in \mathcal{S}_\mV : \, \| G \|_{L^{\overline{q}}(\Omega;V)}\leq 1} \, \lim_{t \to +\infty} \mathbb E\!\left[\left\langle F,G-P_t^\mV G \right\rangle_\mV\right] \\
= & \sup_{G \in \mathcal{S}_\mV : \, \| G \|_{L^{\overline{q}}(\Omega;V)}\leq 1}\mathbb E\!\left[\langle F, G - \E[G] \rangle_\mV\right],
\end{align*}
where in the mid passage we have used the fact that $(P_t^\mV)_{t \ge 0}$ is a symmetric semigroup on $L^2(\Omega;\mV)$. Arguing as in Step $1$, we infer that
\begin{equation*}
\mathbb E\!\left[\langle F, G - \E[G] \rangle_\mV\right] = \lim_{s \to +\infty} \int_0^{s} \mathbb E\!\left[\left\langle D F , D \!\left(P_t^\mV G\right) \right\rangle_{\mH \otimes \mV}\right]dt \, ,
\end{equation*}
for all $ G \in \mathcal{S}_\mV $ with $\|G\|_{L^{\overline{q}}(\Omega;\mV)}\leq 1$. From Proposition \ref{prop:stima_derivata_funzione_vett} and H\"older's inequality it follows that, for all $s>0$,
\begin{align*}
& \left| \int_0^{s} \mathbb E\!\left[\left\langle D F , D \!\left(P_t^\mV G\right) \right\rangle_{\mH \otimes \mV}\right]dt \right| \\
\leq & \left( \mathbb E \! \left[ \|D F \|_{\mathcal H \otimes \mV}^q\right]\right)^{\frac 1 q} \int_0^{s} \left( \mathbb E \! \left[\left\|D \! \left( P_t^\mV G \right) \right\|_{\mathcal H \otimes \mV}^{\overline{q}}\right]\right)^{1/{\overline{q}}} dt 
 \\
\leq & \left( \mathbb E \! \left[ \|D F \|_{\mathcal H \otimes \mV}^q\right]\right)^{\frac 1 q}  \left( \mathbb E \! \left[ \|G \|_{\mV}^{\overline q}\right]\right)^{ 1 / \overline{q}} \int_0^{s}\frac{e^{-t}}{\sqrt{1-e^{-2t}}} \, dt \\
\le  & \, \frac{\pi}{2} \left( \mathbb E \! \left[ \|D F \|_{\mathcal H \otimes \mV}^q\right]\right)^{\frac 1 q} .
\end{align*}
Hence, upon letting $ s \to +\infty $, we end up with the inequality 
\begin{align*}
\left(\mathbb E \! \left[ \left\|F-\mathbb E[F] \right\|_\mV^q\right]\right)^{\frac 1 q} 
\leq \frac{\pi}{2} \left( \mathbb E \! \left[ \|D F \|_{\mathcal H \otimes \mV}^q\right]\right)^{\frac 1 q} ,
\end{align*}
and since the multiplying constant $ \pi / 2$ does not depend on $q$, the latter also holds for $q=1$.
\end{proof}


\section{Norms of expected Malliavin derivatives:\\ proof of Lemma \ref{prop2fd}}\label{sec_mean}

We start by recalling a fundamental result on the equivalence of $ L^q $ norms in Wiener chaoses. We refer the reader to Appendix \ref{appendixA} for a brief introduction to Wiener chaos, and to \cite{Nua95,NP12} for a more complete discussion.  

\begin{lemma}[Corollary 2.8.14 in \cite{NP12}]\label{lem_equiv_chaos}
Let $ F_\ell $ be any element of the $\ell$-th Wiener chaos $C_\ell$ ($\ell\in \mathbb N_{\ge 1}$) of the isonormal Gaussian process $W$. Then, for any $1<s<r<\infty$, we have 
\begin{equation}\label{mau1}
    \left\| F_\ell \right\|_{L^s(\Omega)}\le \left\| F_\ell \right\|_{L^r(\Omega)} \le \left ( \frac{r-1}{s-1}\right )^{\frac \ell 2} \left\| F_\ell \right\|_{L^s(\Omega)}.
\end{equation}
\end{lemma}

\smallskip
\begin{proof}[Proof of Lemma \ref{prop2fd}] 
Let us assume first that $\mathcal V \equiv \mathbb R$ and thus $ F \in \mathcal{S} $. As recalled in Appendix \ref{appendixA}, for any $ \ell \in \mathbb N_{\ge 1}$ we have the identity (cf.~(\ref{equality}))
\begin{equation}\label{mau-bis}
   \left\| \mathbb E \! \left[D^\ell F \right] \right\|_{\mathcal H^{\otimes k}} = \sqrt{\ell!} \left\| J_\ell F \right\|_{L^2(\Omega)},
\end{equation}
where $J_\ell : L^2(\Omega)\to C_\ell$ stands for the orthogonal projection onto the $\ell$-th Wiener chaos $C_\ell$. Plainly, if $ q \in [2,\infty) $ it holds 
\begin{equation}\label{mau2}
    \left\| J_\ell F \right\|_{L^2(\Omega)} \le \| F\|_{L^2(\Omega)} \le \| F\|_{L^q(\Omega)}  \, .
\end{equation}
If, on the contrary, $ q \in (1,2) $, thanks to the definition of orthogonal projection and H\"older's inequality we obtain  
$$
\begin{aligned}
   \left\|J_\ell F \right\|^2_{L^{ 2}(\Omega)} = \mathbb E \! \left[F J_\ell F\right] \le \left\| F \right\|_{L^q(\Omega)} \left\|J_\ell F \right\|_{L^{\overline q}(\Omega)}
    \le \left\| F \right\|_{L^q(\Omega)} \left ( \overline q-1\right )^{\frac {\ell}{2}} \left\|J_\ell F \right\|_{L^{ 2}(\Omega)} ,
\end{aligned}
$$
where the last inequality follows from Lemma \ref{lem_equiv_chaos} with $ F_\ell = J_\ell F  $, $s=2$ and $ r=\overline{q} $. Hence,
\begin{equation}\label{mat1}
     \left\| J_\ell F \right\|_{L^{2}(\Omega)} \le  \left ( \overline q-1\right )^{\frac \ell 2} \left\|F\right\|_{L^q(\Omega)}.
\end{equation}
Substituting \eqref{mau2} (resp.~\eqref{mat1}) into \eqref{mau-bis} for $q\in [2,\infty)$  (resp.~$q\in (1,2)$), we deduce \eqref{stima_exp} in the special case $\mathcal V \equiv \mathbb R$. Let us now turn to the general case of a separable Hilbert space $ \mV $ and $F=\sum_{j=1}^J F_j \, v_j\in \mathcal S_{\mathcal V}$, with  $F_j\in \mathcal S$ and $v_j\in \mathcal V$ for every $ j=1,\ldots,J $. If $q\in [2,\infty)$, by virtue of \eqref{mau-bis} applied componentwise we infer that
\begin{align*}
    \left\| \mathbb E\!\left[D^\ell F\right] \right\|^2_{\mathcal H^{\otimes \ell}\otimes \mathcal V} = \sum_{j=1}^J \left\| \mathbb E \! \left[D^\ell F_j\right] \right\|^2_{\mathcal H^{\otimes \ell}} = \ell! \sum_{j=1}^J \mathbb E\!\left[\left(J_\ell F_j\right)^2\right] \le & \, \ell! \sum_{j=1}^J \mathbb E\!\left[F_j^2\right] \\ = & \, \ell! \left\| F \right\|_{L^2(\Omega;\mathcal V)}^2 \\ \le & \, \ell! \left\| F \right\|_{L^q(\Omega;\mathcal V)}^2 .
\end{align*}
On the other hand, if $q\in (1,2)$, from \eqref{mau-bis} and \eqref{mat1} applied componentwise we obtain 
\begin{align*}
\left\| \mathbb E\!\left[D^\ell F\right] \right\|^2_{\mathcal H^{\otimes \ell}\otimes \mathcal V} = \ell! \sum_{j=1}^J \mathbb E\!\left[\left(J_\ell F_j\right)^2\right] \le & \, \ell!\left ( \overline q-1\right )^{\ell} \sum_{j=1}^J \|F_j\|^2_{L^q(\Omega)} \\
\le & \, \ell!\left ( \overline q-1\right )^{\ell} \|F\|^2_{L^q(\Omega;\mathcal V)} \, ,
\end{align*}
where the last passage follows thanks to Minkowski's integral inequality:
\begin{align*}
\left( \sum_{j=1}^J \left\| F_j \right\|^2_{L^q(\Omega)}\right)^{\frac q 2} = \left (\sum_{j=1}^J \left(\mathbb E \!\left[\left|F_j\right|^q\right]\right)^{\frac 2 q}\right)^{\frac q 2}\le \mathbb E\!\left [ \left ( \sum_{j=1}^J  \left(\left|F_j\right|^q\right )^{\frac 2 q}\right )^{\frac q 2}\right]  = & \, \mathbb E \! \left [ \left ( \sum_{j=1}^J  F_j^2\right )^{\frac q 2}\right ] \\ = & \left\| F \right\|_{L^q(\Omega;\mV)}^q .
\end{align*}
The proof of Lemma \ref{prop2fd} is therefore complete. 
\end{proof}

\begin{remark}[Equivalent norms in $C_\ell$ for $ s=1 $]\label{equivalence-1}\rm
In fact the multiplying constant in \eqref{mau1} is not optimal as $s \downarrow 1 $, since it blows up, whereas by duality one can show that the equivalence of norms actually holds for $ s=1 $ as well. Indeed, by H\"older's inequality and \eqref{mau1} (with $s=2$ and $r=3$) we have 
$$
\left\| F_\ell \right\|_{L^2(\Omega)} \le \left\|F_\ell \right\|_{L^1(\Omega)}^{\frac 1 4} \left\| F_\ell \right\|_{L^3(\Omega)}^{\frac 3 4} \le 2^{\frac{3\ell}{8}} \left\| F_\ell \right\|_{L^1(\Omega)}^{\frac 1 4} \left\| F_\ell \right\|_{L^2(\Omega)}^{\frac 3 4} ,
$$
that is
$$
\left\| F_\ell \right\|_{L^2(\Omega)} \le 2^{\frac{3\ell}{8}} \left\| F_\ell \right\|_{L^1(\Omega)} .
$$
Hence, the $ L^1(\Omega) $ norm is equivalent to the $ L^2(\Omega) $ norm in $ C_\ell$, and therefore to any other $ L^r(\Omega) $ norm for all $ r \in (1,\infty) $. Note that the equivalence of norms for all $ r,s \in [1,\infty) $ had already been observed in \cite[Proposition 3.1]{Maa10}, although with no quantitative constants due to the abstract setting addressed therein. Nevertheless, as the reader may have noticed, in the proof of Lemma \ref{prop2fd} we only used \eqref{mau1} for $ r,s \in [2,\infty) $; what prevents it from working in the case $ q=1 $ is the fact that the multiplying constant in \eqref{mau1} (necessarily) blows up as $ r \to \infty $. On top of that, regardless of the specific proof, it is straightforward to construct one-dimensional counterexamples showing that $ J_\ell $ cannot be extended continuously to the whole $ L^1(\Omega) $.
	\end{remark}

 \section{Norms of (expected) Malliavin derivatives:\\ proofs of Lemmas \ref{prop2} and \ref{lem1fd}}\label{sec_propmean}
 
In the case $ q>1 $ the strategy we developed above strongly relies on Lemma \ref{prop2fd}, which however fails for $ q=1 $. Here we will follow a different approach based on \emph{one-dimensional} Sobolev inequalities, which turns out to be very effective to address the finite-dimensional case as well. The starting point is Lemma 5.4 in \cite{AF03}, which asserts that for any $\rho>0$ and any $g\in C^2([0,\rho])$ we have
\begin{align}\label{adams-est-1}
\left|g'(0)\right| \leq &  \, 9\left( \rho^{-2}\int_0^\rho\left|g(t)\right|dt + \int_0^\rho \left|g''(t)\right|dt \right) .
\end{align}

We first extend \eqref{adams-est-1} to the case of $ \mV $-valued functions and all $q \ge 1$, then deduce a further generalization to functions whose derivatives are integrable with respect to the Gaussian measure, obtaining an analogue of \cite[Lemma 5.5]{AF03} (in dimension $1$).
\begin{lemma}\label{lem:adams-est-2}
Let $ q \in [1,\infty) $. Then, for any $ \rho>0 $ and all $ g \in C^2([0,\rho];\mV) $, we have 
\begin{equation}\label{adams-est-2}
	\left\|g'(0) \right\|_{\mV}^q \leq 9^q \, 2^{q-1}  \left(\rho^{-q-1}\int_0^\rho\left\|g(t)\right\|^q_{\mV}dt + \rho^{q-1} \int_0^\rho \left\|g''(t)\right\|_{\mV}^q  dt \right) .
\end{equation}
	\end{lemma}
\begin{proof}
It is enough to observe that, given an arbitrary $ v \in \mV $ with $ \| v \|_{\mV} = 1 $, the function 
	$$
	t \mapsto \left\langle g(t) , v \right\rangle_{\mV}
	$$
	belongs to $ C^2([0,\rho]) $, so that by applying to it estimate \eqref{adams-est-1} we obtain 
	$$
	\left| \left\langle g'(0) , v \right\rangle_{\mV} \right| \leq  9\left(\rho^{-2}\int_0^\rho\left|  \left\langle g(t) , v \right\rangle_{\mV} \right|dt+\int_0^\rho \left| \left\langle g''(t) , v \right\rangle_{\mV}\right|dt\right) ,
	$$
	which yields
	$$
 \left\| g'(0) \right\|_{\mV} = \sup_{v \in \mV : \, \| v \|_{\mV}=1} \left| \left\langle g'(0) , v \right\rangle_{\mV} \right|	\le 9\left(\rho^{-2}\int_0^\rho \left\| g(t) \right\|_{\mV} dt+\int_0^\rho \left\| g''(t) \right\|_{\mV} dt\right) ,
	$$
	namely \eqref{adams-est-2} for $q=1$. The case $ q>1 $ follows by H\"older's inequality and the convexity inequality $ (\alpha+\beta)^q \le 2^{q-1}(\alpha^q+\beta^q)  $ (let $ \alpha,\beta \ge 0 $), since 
	$$
	\begin{aligned}
	 \left\| g'(0) \right\|_{\mV}^q  \le & \, 9^q\left[\rho^{-2} \rho^{\frac{q-1}{q}} \left(\int_0^\rho \left\| g(t) \right\|_{\mV}^q dt \right)^{\frac 1 q}+ \rho^{\frac{q-1}{q}} \left( \int_0^\rho \left\| g''(t) \right\|_{\mV}^q dt \right)^{\frac 1 q} \right]^q \\
	 \le & \, 9^q \, 2^{q-1}  \left(\rho^{-q-1}\int_0^\rho\left\|g(t)\right\|^q_{\mV}dt + \rho^{q-1} \int_0^\rho \left\|g''(t)\right\|_{\mV}^q  dt \right) .
	 \end{aligned}
	$$
	The proof is therefore complete.
	\end{proof}

\begin{lemma} \label{lem:gen_adams_gaussian}
Let $ q \in [1,\infty) $. Then, for any $ \rho \in (0,1) $ and all $ f \in S(\R;\mV) $, we have 
\begin{align} \label{stima_adams_n=1_m=2}
\int_{\R} \left\| f'(x) \right\|_{\mV}^q \gamma_1(dx)
\leq & 18^q \sqrt e \left( \rho^{-q} \int_{\R} \left\| f(x) \right\|_{\mV}^q \gamma_1(dx) +  \rho^{q} \int_{\R} \left\| f''(x) \right\|_{\mV}^q \gamma_1(dx) \right).
\end{align}
\end{lemma}
\begin{proof}
Given an arbitrary $x\in \R$, we consider the function
\begin{align*}
\Phi(t):=f(x-t) \qquad \forall t\in\R \, .
\end{align*}
Since $\Phi\in C^2([0,\rho];\mV)$ for any $ \rho>0 $, by Lemma \ref{lem:adams-est-2} we have 
\begin{align*}
\left\|f'(x)\right\|_\mV^q = \left\|\Phi'(0)\right\|_\mV^q \leq 9^q \, 2^{q-1} \left( \rho^{-q-1}\int_0^\rho \left\|f(x-t)\right\|_\mV^q dt + \rho^{q-1} \int_0^\rho \left\|f''(x-t)\right\|_\mV^q dt \right) \! .    
\end{align*}
We now integrate between $0$ and $+\infty$, with respect to $\gamma_1(dx)$, both sides of the above inequality, obtaining
\begin{align*}
& \, \int_0^{+\infty} \left\|f'(x)\right\|_\mV^q \gamma_1(dx) \\
\leq & \, \frac{9^q \, 2^{q-1}}{\sqrt{2\pi}} \int_0^\rho \left[ \int_0^{+\infty} \left( \rho^{-q-1} \left\|f(x-t)\right\|_\mV^q + \rho^{q-1} \left\|f''(x-t)\right\|_\mV^q \right)e^{-\frac{x^2}{2}}dx\right] dt \, .
\end{align*}
The change of variables $y=x-t$ yields
\begin{align*}
&\int_0^\rho \left[ \int_0^{+\infty} \left( \rho^{-q-1} \left\|f(x-t)\right\|_\mV^q + \rho^{q-1} \left\|f''(x-t)\right\|_\mV^q \right)e^{-\frac{x^2}{2}}dx\right] dt\\
= & \int_0^\rho e^{-\frac{t^2}{2}}\left[ \int_{-t}^{+\infty} \left( \rho^{-q-1} \left\|f(y)\right\|_\mV^q + \rho^{q-1} \left\|f''(y)\right\|_\mV^q \right) e^{-\frac{y^2}{2}-yt} \, dy \right] dt \, .
\end{align*}
Since $y>-t$, it follows that $e^{-ty}\leq e^{t^2}$, which gives
\begin{align*}
&  \int_0^\rho e^{-\frac{t^2}{2}}\left[ \int_{-t}^{+\infty} \left( \rho^{-q-1} \left\|f(y)\right\|_\mV^q + \rho^{q-1} \left\|f''(y)\right\|_\mV^q \right) e^{-\frac{y^2}{2}-yt} \, dy \right] dt \\
\leq & \int_0^\rho e^{\frac{t^2}{2}}\left[\int_{-t}^{+\infty} \left( \rho^{-q-1} \left\|f(y)\right\|_\mV^q + \rho^{q-1} \left\|f''(y)\right\|_\mV^q \right) e^{-\frac{y^2}{2}}dy\right]dt \\
\leq & \int_0^\rho e^{\frac{t^2}{2}}\left[\int_{-\infty}^{+\infty} \left( \rho^{-q-1} \left\|f(y)\right\|_\mV^q + \rho^{q-1} \left\|f''(y)\right\|_\mV^q \right) e^{-\frac{y^2}{2}}dy\right] dt \\
\leq & \, e^{\frac 1 2} \left[\int_{-\infty}^{+\infty} \left( \rho^{-q} \left\|f(y)\right\|_\mV^q + \rho^{q} \left\|f''(y)\right\|_\mV^q \right) e^{-\frac{y^2}{2}}dy\right] , 
\end{align*}
recalling that $ \rho<1 $. Hence, we end up with 
\begin{equation} \label{stima_lemma_+infty}
\int_0^{+\infty} \left\|f'(x)\right\|_\mV^q \gamma_1(dx)
\leq 9^q \, 2^{q-1} \sqrt e\int_{\R} \left( \rho^{-q} \left\|f(x)\right\|_\mV^q + \rho^{q} \left\|f''(x)\right\|_\mV^q \right) \gamma_1(dx) \, .
\end{equation}
Finally, by applying \eqref{stima_lemma_+infty} to $ f(-x) $ and summing up the two inequalities, using the fact that $ \gamma_1 $ is even, we deduce \eqref{stima_adams_n=1_m=2}. 
\end{proof}

\begin{remark}[A stronger bound that fails]\rm
In \cite[Lemma 5.5]{AF03} it was proved that
\begin{align*}
\left\| \nabla{f} \right\|_{L^q(E)}\leq K \left(\rho^{-1} \left\| f \right\|_{L^q(E)} + \rho \left\| \nabla^2 f \right\|_{L^q(E)} \right) \qquad \forall  f\in W^{2,q}(E) \, ,
\end{align*}
where $E\subseteq \R^n$ is an open domain which satisfies the so-called \emph{cone condition}, $K$ is a positive constant depending on $ E , q$ and $\rho$ ranges in the interval $(0,\rho_0)$, the quantity $\rho_0>0$ being the maximum \emph{height} of the cone related to $E$ (we refer to \cite[Definitions 4.4, 4.5 and 4.6]{AF03} for the precise notions). In particular, if $E=\R^n$ the above inequality is satisfied for any $\rho>0$ (see \cite[Remark 5.7]{AF03}). One can ask if the same can be true for the standard Gaussian measure on the whole $\R$, i.e., if there exists a positive constant $K$ such that 
\begin{align}
\label{falsa_dis_n=1}
\left\| f' \right\|_{L^q(\R,\gamma_1)} \leq K \left( \rho^{-1} \left\| f \right\|_{L^q(\R,\gamma_1)}  + \rho \left\| f'' \right\|_{L^q(\R,\gamma_1)}\right) \qquad \forall f \in S(\R) \, ,
\end{align}
for any $\rho>0$. Unfortunately, the answer is negative. Indeed, let us assume by contradiction that \eqref{falsa_dis_n=1} holds, and pick $f(x)=x \in S(\R) $. Then $f'(x) \equiv 1$ and $f''(x)\equiv0$, so that we would end up with
\begin{align*}
1 \leq \frac{\sqrt 2 \, K}{\sqrt \pi \, \rho} \qquad  \forall \rho>0 \, ,    
\end{align*}
which is absurd. This is the main reason why we are not able to iterate \eqref{d2} backwards.
\end{remark}
 
\begin{proof}[Proof of Lemma \ref{prop2}] We split the proof into two cases: first we deal with real-valued functions and $ \ell=1 $, then extend the thesis to general $ \mV $-valued functions and all $ \ell \in \N_{\ge 1} $. 

\smallskip
{\bf Case $\boldsymbol{\ell = 1}$.} Given any $ F \in \mathcal{S} $, by definition we know that there exist $ m \in \mathbb{N}_{\ge 1} $ and $ f \in S(\mathbb{R}^m) $ such that $F=f(W(h_1),\ldots,W(h_m))$. Without loss of generality, we may assume that $\{ h_i \}_{i=1,\ldots,m} $ are orthonormal vectors in $\mathcal H$. Our aim is to estimate  $\|\mathbb E[DF]\|_{\mathcal H}$. Keeping in mind Remark \ref{equiv-ref}, we notice that 
$$
\mathbb E[DF]= \sum_{i=1}^m z_i \,  h_i  \, ,
$$
where
\begin{align*}
\alpha_i := \int_{\R^m} \frac{\partial f}{\partial x_i}(x) \, \gamma_m(dx)  \quad i=1,\ldots, m \, , \qquad z:=(\alpha_1,\ldots,\alpha_m) \, .
\end{align*}
Hence,
\begin{align}
\label{stima_derF_finale_1}
\left\| \mathbb E[DF] \right\|_{\mathcal H}
=|(\alpha_1,\ldots,\alpha_m)|_{\R^m}
=\langle z,u\rangle_{\R^m} \, ,
\end{align}
for some normal vector $u \in \R^m$. From the definition of $z$, we have:
\begin{equation}\label{stima_derF_finale_2}
\begin{aligned}
\langle z,u\rangle_{\R^m}
=  \sum_{i=1}^m\int_{\R^m}
\frac{\partial f}{\partial x_i}(x)\, \gamma_m(dx)  \cdot u_i
= & \int_{\R^m} \left( \sum_{i=1}^m \frac{\partial f}{\partial x_i}(x) \, u_i \right)  \gamma_m(dx) \\
= & \int_{\R^m} \left\langle \nabla f(x),u \right\rangle_{\R^m}  \gamma_m(dx) \, .
\end{aligned}
\end{equation}
Let us now consider an orthonormal basis $\{z_i\}_{i=1,\ldots,m}$ such that $z_1=u$, along with the rotation $R:\R^m \rightarrow \R^m$ which maps it into the canonical basis $\{e_i\}_{i=1,\ldots,m}$. We denote by $(R_{ij})_{i,j=1,\ldots,m}$ the matrix associated to the transformation $R$. By means of the change of variables $x=R^{T}y$, we can infer that 
\begin{equation} \label{cambio_variabili_1}
\begin{aligned}
\int_{\R^m} \left\langle \nabla f(x),u \right\rangle_{\R^m}  \gamma_m(dx)
= & \int_{\R^m} \left\langle \nabla f \! \left(R^{T}y\right)\!,u \right\rangle_{\R^m}\gamma_m\!\left(R^T dy\right) \\
= & \int_{\R^m} \left\langle \nabla f \! \left(R^{T}y\right)\!,u \right\rangle_{\R^m}\gamma_m(dy) \, ,
\end{aligned}
\end{equation}
thanks to the rotational invariance of $ \gamma_m $. Let us focus on the function under the integral sign. We have:
\begin{align*}
u=z_1= R^Te_1 =\sum_{i=1}^m R_{1i}\, e_i \, ,
\end{align*}
which gives
\begin{align}
\label{cambio_variabili_2}
\left\langle \nabla f \! \left(R^{T}y\right)\!,u \right\rangle_{\R^m}
= \sum_{i=1}^m \frac{\partial f}{\partial x_i}\!\left(R^Ty\right) R_{1i} \qquad \forall y\in \R^m \, .
\end{align}
If we further introduce the function $\tilde f:\R^m\rightarrow \R$ defined by
\begin{align*}
\tilde f(y):=f\!\left(\sum_{i=1}^m R_{j1} \, y_j, \ldots,\sum_{j=1}^m R_{jm} \, y_j\right)=f\!\left(R^Ty\right) \qquad \forall y \in \R^m \, ,
\end{align*}
it follows that
\begin{align}
\label{cambio_variabili_3}
\frac{\partial \tilde f}{\partial y_j}(y)=\sum_{i=1}^m\frac{\partial f}{\partial x_i}\!\left(R^Ty\right)R_{ji} \quad j=1,\ldots,m \, , \qquad \forall y \in \mathbb{R}^m \, .    
\end{align}
By combining \eqref{cambio_variabili_1}, \eqref{cambio_variabili_2} and \eqref{cambio_variabili_3}, we thus infer that
\begin{equation}
\label{stima_derF_finale_3}
\begin{aligned}
\int_{\R^m} \left\langle \nabla f(x),u \right\rangle_{\R^m}  \gamma_m(dx)
= & \int_{\R^m} \frac{\partial\tilde f}{\partial y_1}(y) \, \gamma_m(dy) \\
 = & \int_{\R^{m-1}} \left( \int_\R \frac{\partial\tilde f}{\partial y_1}(y_1,\tilde{y}) \, \gamma_1(dy_1) \right) \gamma_{m-1}(d\tilde{y}) \, ,
\end{aligned}
\end{equation}
where $ \tilde{y}:=(y_2,\ldots,y_m) \in \R^{m-1} $, by virtue of the product structure of $ \gamma_m $. If we apply \eqref{stima_adams_n=1_m=2} (with $ q=1 $ and $ \mV \equiv \R $) to the function $ y_1 \mapsto f(y_1,\tilde{y})$, for any fixed $\tilde{y} \in \R^{m-1} $, we get 
\begin{equation}\label{stima_der_prima_unidim}
\int_\R \left|\frac{\partial\tilde f}{\partial y_1}(y_1,\tilde{y}) \right| \gamma_1(dy_1)
\le 18 \sqrt{e} \left( \rho^{-1}\int_{\R}\left|\tilde f(y_1,\tilde{y})\right| \gamma_1(dy_1)  + \rho
\int_{\R}\left|\frac{\partial^2 \tilde f}{\partial y_1^2}(y_1,\tilde{y})\right|\gamma_1(dy_1) \right) 
\end{equation}
for all $\rho\in(0,1)$. Integrating both sides with respect to $ \gamma_{m-1}(d\tilde{y}) $, we deduce that
\begin{equation}\label{stima_der_prima_unidim-bis}
\begin{aligned}
\int_{\R^m} \left|\frac{\partial \tilde f}{\partial y_1}(y)\right| \gamma_m(dy)
\leq & \, 18\sqrt{e}\left(\rho^{-1}\int_{\R^{m}}\left|\tilde f(y)\right| \gamma_m(dy) + \rho
\int_{\R^{m}}\left|\frac{\partial^2 \widetilde f}{\partial y_1^2}(y)\right| \gamma_m(dy) \right)\\
\leq & \, 18\sqrt e\left( \rho^{-1} \, \big\|\tilde f \big\|_{L^1(\R^m,\gamma_m)} + \rho \, \big\| \nabla^2 \tilde f \big\|_{L^1(\R^m,\gamma_m;\R^{m}\otimes \R^m)} \right)
\end{aligned}
\end{equation}
for all $\rho\in(0,1)$. To conclude, we notice that, by using the reverse change of variables $ y=Rx $, one easily obtains
$$
\big\|\tilde f \big\|_{L^1(\R^m,\gamma_m)} = \left\|  f \right\|_{L^1(\R^m,\gamma_m)} = \left\| F \right\|_{L^1(\Omega)}
$$
and 
$$
\big\| \nabla^2 \tilde f \big\|_{L^1(\R^m,\gamma_m;\R^{m}\otimes \R^m)}  = \big\| \nabla^2 f \big\|_{L^1(\R^m,\gamma_m;\R^{m}\otimes \R^m)} = \left\| D^2 F \right\|_{L^1(\Omega;\mH^{\otimes 2})} .
$$
Therefore, from \eqref{stima_derF_finale_1}, \eqref{stima_derF_finale_2}, \eqref{stima_derF_finale_3}, \eqref{stima_der_prima_unidim} and \eqref{stima_der_prima_unidim-bis}, we end up with 
\begin{align}
\label{stima_media_derprima}
\left\|\mathbb E[DF] \right\|_{\mathcal H}
\leq 18 \sqrt e \left(\rho^{-1} \left\|F\right\|_{L^1(\Omega)} + \rho \left\|D^2F\right\|_{L^1(\Omega;\mathcal H^{\otimes 2})} \right)
\end{align}
for all $\rho\in(0,1)$, which yields \eqref{d2} for $ \mV \equiv \R $ and $ \ell=1 $.

\smallskip
{\bf Case $ \boldsymbol{ \ell > 1 } $.} First of all, we observe that it is enough to extend \eqref{stima_media_derprima} to $ \mV $-valued functions, namely 
\begin{equation}
\label{stima_media_derprima-V}
\left\|\mathbb E[DF] \right\|_{\mathcal H \otimes \mV }
\leq 18 \sqrt e \left(\rho^{-1} \left\|F\right\|_{L^1(\Omega;\mV)} + \rho \left\|D^2F\right\|_{L^1(\Omega;\mathcal H^{\otimes 2} \otimes \mV )} \right) ,
\end{equation}
so that \eqref{d2} will follow upon replacing $ F $ with $ D^{\ell-1} F $ and $  \mV $ with $ \mH^{\ell-1} \otimes \mV $. Given any $ F \in \mathcal{S}_\mV $ of the form 
$$
F = \sum_{j=1}^J F_j \, v_j \, ,
$$
with $J\in \mathbb N_{\ge 1}$, $F_j \in \mathcal{S}$ and $v_j\in \mV$ for every $ j=1 , \ldots , J $, we have:
\begin{equation}
\label{stima_media_derprima-V-1}
\left\|\mathbb E[DF] \right\|_{\mathcal H \otimes \mV } = \left( \sum_{j=1}^J \left\| \E \! \left[ D F_j \right] \right\|_{\mH}^2 \right)^{\frac 1 2} ,
\end{equation}
provided the vectors $ \{ v_j \}_{j=1,\ldots,J} $ are orthonormal in $\mV$ (which can be assumed with no loss of generality). By applying \eqref{stima_media_derprima} to each component $ F_j $, it follows that 
\begin{equation}
\label{stima_media_derprima-V-2}
\begin{aligned} 
& \left( \sum_{j=1}^J \left\| \E \! \left[ D F_j \right] \right\|_{\mH}^2 \right)^{\frac 1 2} \\
\le \, & 18 \sqrt{2e} \left( \rho^{-2} \sum_{j=1}^J \left\|F_j\right\|_{L^1(\Omega)}^2 + \rho^2 \sum_{j=1}^J \left\|D^2F_j\right\|_{L^1(\Omega;\mathcal H^{\otimes 2} )}^2 \right)^{\frac 1 2} \\
\le \, & 18 \sqrt{2e} \left[ \rho^{-1} \left( \sum_{j=1}^J\left( \E\!\left[|F_j|\right] \right)^2 \right)^{\frac 1 2} + \rho \left( \sum_{j=1}^J \left( \E  \!\left[\left\|D^2 F_j \right\|_{\mH^{\otimes 2}} \right] \right)^2 \right)^{\frac 1 2}  \right] .
\end{aligned}
\end{equation}
In order to conclude, we can take advantage again of Minkowski's integral inequality, which entails 
\begin{equation}
\label{stima_media_derprima-V-3}
\left( \sum_{j=1}^J\left( \E\!\left[|F_j|\right] \right)^2 \right)^{\frac 1 2} \le \E\!\left[ \left( \sum_{j=1}^J F_j^2 \right)^{\frac 1 2} \right] = \left\| F \right\|_{L^1(\Omega;\mV)}
\end{equation}
and
\begin{equation}
\label{stima_media_derprima-V-4}
 \left( \sum_{j=1}^J \left( \E  \!\left[\left\|D^2 F_j \right\|_{\mH^{\otimes 2}} \right] \right)^2 \right)^{\frac 1 2} \le\E\!\left[ \left( \sum_{j=1}^J \left\|D^2 F_j \right\|_{\mH^{\otimes 2}}^2 \right)^{\frac 1 2} \right] = \left\| D^2 F \right\|_{L^1(\Omega;\mH^{\otimes 2}\otimes \mV)} .
\end{equation}
Inequality \eqref{stima_media_derprima-V} is therefore established, as a consequence of \eqref{stima_media_derprima-V-1}, \eqref{stima_media_derprima-V-2}, \eqref{stima_media_derprima-V-3} and \eqref{stima_media_derprima-V-4}. 
\end{proof}

Finally, we are able to prove Lemma \ref{lem1fd} by extending \eqref{stima_adams_n=1_m=2} to any Euclidean dimension $  n $ and any order of derivative $ \ell $.

\smallskip
\begin{proof}[Proof of Lemma \ref{lem1fd}] 
As above, we will first address the special case $ \ell=1 $ and then deal with a general $ \ell>1 $ by induction.

\smallskip
{\bf Case $ \boldsymbol{\ell=1} $.} First of all we observe that, similarly the proof of Lemma \ref{prop2}, proving \eqref{dav1} when $ \ell=1 $ and $ \operatorname{dim}(\mH)=n<\infty $ is equivalent to proving that
	\begin{equation}\label{dav1-proof}
	\begin{gathered}
	\left\| \nabla f \right\|_{L^q(\R^n , \gamma_n ; \R^n \otimes \mathcal V)} \le C_{1,n} \left( \rho^{-1} \, \| f \|_{L^q(\R^n,\gamma_n;\mathcal V)} + \rho \left\| \nabla^2 f \right\|_{L^q(\R^n,\gamma_n; (\R^n)^{\otimes 2} \otimes \mathcal V)} \right) \\
	 \forall \rho \in (0,1) \, ,
\end{gathered}
	\end{equation}
for all $ f \in S(\R^n;\mV) $ and any $ q \in [1,\infty) $. To this aim, note that by convexity and the product structure of $ \gamma_n $ we have
\begin{equation}\label{dav1-proof-2}
\begin{aligned}
\left\| \nabla f \right\|_{L^q(\R^n , \gamma_n ; \R^n \otimes \mathcal V)}^q = & \int_{\R^n} \left( \sum_{i=1}^n  \left\| \frac{\partial f}{\partial x_i}(x) \right\|^2_{\mV} \right)^{\frac q 2} \gamma_n(dx) \\
\le & \, n^{\frac{(q-2)^+}{2}} \int_{\R^n} \left( \sum_{i=1}^n  \left\| \frac{\partial f}{\partial x_i}(x) \right\|^q_{\mV} \right) \gamma_n(dx) \\
= & \,  n^{\frac{(q-2)^+}{2}} \sum_{i=1}^n \int_{\R^{n-1}} \left( \int_\R \left\| \frac{\partial f}{\partial x_i}(x_i,\tilde{x}_i) \right\|^q_{\mV} \gamma_1(dx_i) \right) \gamma_{n-1}(d\tilde{x}_i) \, ,
\end{aligned}
\end{equation}
where we let $ \tilde{x}_i $ denote the $(n-1)$-dimensional vector containing all the components of $ x=(x_1,\ldots,x_n) $ except $ x_i $ and, with some abuse of notation, we write $ x \equiv (x_i,\tilde{x}_i) $. We are therefore in position to apply \eqref{stima_adams_n=1_m=2} to the innermost integral for every $ i=1,\ldots,n $, which yields (for all $ \rho \in (0,1) $)
\begin{equation}\label{dav1-proof-3}
\begin{aligned}
& \sum_{i=1}^n \int_{\R^{n-1}} \left( \int_\R \left\| \frac{\partial f}{\partial x_i}(x_i,\tilde{x}_i) \right\|^q_{\mV} \gamma_1(dx_i) \right) \gamma_{n-1}(d\tilde{x}_i) \\
 \le & \, 18^q \sqrt e \, \sum_{i=1}^n \left( \rho^{-q} \int_{\R^n} \left\| f(x) \right\|_{\mV}^q \gamma_n(dx) +  \rho^{q} \int_{\R^n} \left\| \frac{\partial^2 f}{\partial x_i^2}(x) \right\|_{\mV}^q \gamma_n(dx) \right) \\
 = & \, 18^q \sqrt e \left[ n \, \rho^{-q} \int_{\R^n} \left\| f(x) \right\|_{\mV}^q \gamma_n(dx) +  \rho^{q} \int_{\R^n} \left( \sum_{i=1}^n \left\| \frac{\partial^2 f}{\partial x_i^2}(x) \right\|_{\mV}^q \right) \gamma_n(dx) \right] \\
 \le & \, 18^q \sqrt e \left[ n \, \rho^{-q} \int_{\R^n} \left\| f(x) \right\|_{\mV}^q \gamma_n(dx) + n^{\frac{(2-q)^+}{2}} \rho^{q} \int_{\R^n} \left( \sum_{i=1}^n \left\| \frac{\partial^2 f}{\partial x_i^2}(x) \right\|_{\mV}^2 \right)^{\frac q 2} \gamma_n(dx) \right] , 
\end{aligned}
\end{equation}
where in the last passage we have used another elementary concavity inequality. By combining \eqref{dav1-proof-2} and \eqref{dav1-proof-3}, we end up with 
$$
\begin{aligned}
& \left\| \nabla f \right\|_{L^q(\R^n , \gamma_n ; \R^n \otimes \mathcal V)}^q \\ \le & \, 18^q \sqrt e \, n^{1\vee \frac{q}{2}} \left( \rho^{-q} \int_{\R^n} \left\| f(x) \right\|_{\mV}^q \gamma_n(dx) + \rho^q \int_{\R^n} \left\| \nabla^2 f (x) \right\|_{(\R^n)^{\otimes 2} \otimes \mathcal V}^q  \gamma_n(dx) \right) ,
\end{aligned}
$$
whence 
$$
\begin{aligned}
\left\| \nabla f \right\|_{L^q(\R^n , \gamma_n ; \R^n \otimes \mathcal V)} \le & \, 18 \, e^{\frac{1}{2q}} \, n^{\frac{1}{q} \vee \frac{1}{2}} \left( \rho^{-1} \, \| f \|_{L^q(\R^n,\gamma_n;\mathcal V)} + \rho \left\| \nabla^2 f \right\|_{L^q(\R^n,\gamma_n; (\R^n)^{\otimes 2} \otimes \mathcal V)} \right) \\
\le & \, 18 \sqrt{e} \, n \left( \rho^{-1} \, \| f \|_{L^q(\R^n,\gamma_n;\mathcal V)} + \rho \left\| \nabla^2 f \right\|_{L^q(\R^n,\gamma_n; (\R^n)^{\otimes 2} \otimes \mathcal V)} \right)
\end{aligned}
$$
recalling that $ q \ge 1 $. We have therefore established \eqref{dav1-proof}. 

\smallskip
{\bf Case $ \boldsymbol{\ell>1} $.} Let $ \ell \ge 2 $ and assume that \eqref{dav1} holds with $ \ell  $ replaced by $ \ell - 1 $. Given $ f \in S(\R^n;\mV) $, $ q \in [1,\infty) $ and $ \rho \in (0,1) $, firstly we can apply \eqref{dav1-proof} with $ f $ replaced by $ \nabla^\ell f $ and $ \mV $ replaced by $ (\R^n)^{\otimes (\ell-1)} \otimes \mV $, which entails
	\begin{equation}\label{dav1-proof-4}
	\begin{aligned}
	&\left\| \nabla^\ell f \right\|_{L^q(\R^n , \gamma_n ; (\R^n)^{\otimes \ell} \otimes \mathcal V)} \\
\le & \, C_{1,n} \left( \rho^{-1} \left\| \nabla^{\ell-1} f \right\|_{L^q(\R^n,\gamma_n;(\R^n)^{\otimes (\ell-1)} \otimes \mathcal V)} + \rho \left\| \nabla^{\ell + 1} f \right\|_{L^q(\R^n,\gamma_n; (\R^n)^{\otimes (\ell+1)} \otimes \mathcal V)} \right) .
\end{aligned}
	\end{equation}
	On the other hand, in view of the induction hypothesis, the first summand on the right-hand side can be bounded by
		\begin{equation}\label{dav1-proof-5}
	\begin{aligned}
	& \left\| \nabla^{\ell-1} f \right\|_{L^q(\R^n,\gamma_n;(\R^n)^{\otimes (\ell-1)} \otimes \mathcal V)} \\
	\le & \, C_{\ell-1,n} \left( \delta^{-\ell + 1} \left\| f \right\|_{L^q(\R^n,\gamma_n;\mV)} + \delta \left\| \nabla^\ell f \right\|_{L^q(\R^n,\gamma_n;(\R^n)^{\otimes \ell}\otimes \mV)} \right) \qquad \forall \delta \in (0,1) \, .
	\end{aligned}
		\end{equation}
Upon choosing 
$$
\delta = \frac{\rho}{2 \, C_{1,n} \, C_{\ell-1,n} } \, ,
$$ 
and plugging \eqref{dav1-proof-5} in \eqref{dav1-proof-4}, we obtain:
	\begin{equation*}
	\begin{aligned}
	\left\| \nabla^\ell f \right\|_{L^q(\R^n , \gamma_n ; (\R^n)^{\otimes \ell} \otimes \mathcal V)}
\le & \, \frac{1}{2} \left\| \nabla^\ell f \right\|_{L^q(\R^n , \gamma_n ; (\R^n)^{\otimes \ell} \otimes \mathcal V)} + 2^{\ell-1} C_{1,n}^\ell \, C_{\ell-1,n}^\ell \, \rho^{-\ell} \left\| f \right\|_{L^q(\R^n,\gamma_n;\mV)} \\
& + C_{1,n} \, \rho \left\| \nabla^{\ell + 1} f \right\|_{L^q(\R^n,\gamma_n; (\R^n)^{\otimes (\ell+1)} \otimes \mathcal V)} ,
\end{aligned}
	\end{equation*}
		from which we can deduce that 
	\begin{equation}\label{dav1-proof-7}
	\begin{aligned}
	& \left\| \nabla^\ell f \right\|_{L^q(\R^n , \gamma_n ; (\R^n)^{\otimes \ell} \otimes \mathcal V)} \\
\le & \, 2^{\ell} \, C_{1,n}^\ell \, C_{\ell-1,n}^\ell \, \rho^{-\ell} \left\| f \right\|_{L^q(\R^n,\gamma_n;\mV)} + C_{1,n} \, \rho \left\| \nabla^{\ell + 1} f \right\|_{L^q(\R^n,\gamma_n; (\R^n)^{\otimes (\ell+1)} \otimes \mathcal V)} \\
\le & \, \underbrace{2^{\ell} \, C_{1,n}^\ell \, C_{\ell-1,n}^\ell}_{=C_{\ell,n}} \left( \rho^{-\ell} \left\| f \right\|_{L^q(\R^n,\gamma_n;\mV)}  + \rho \left\| \nabla^{\ell + 1} f \right\|_{L^q(\R^n,\gamma_n; (\R^n)^{\otimes (\ell+1)} \otimes \mathcal V)} \right) .
	\end{aligned}
		\end{equation}
It is then straightforward to check that if $ C_{\ell-1,n} $ is of the form \eqref{zeta}, so is $ C_{\ell,n} $, and the proof is complete since \eqref{dav1-proof-7} is equivalent to \eqref{dav1}.
\end{proof}

%
%

\appendix 

\section{Wiener chaos}\label{appendixA}
Let $\lbrace H_k \rbrace_{k \in \N}$ be the family of Hermite polynomials, that is $H_0 = 1$ and for $k\in \mathbb N_{\ge 1}$ 
\begin{equation*}
    H_k(t) := (-1)^k \, e^{\frac{t^2}{2}} \frac{d^k}{dt^k} \! \left( e^{-\frac{t^2}{2}} \right) \qquad \forall  t \in \mathbb R \, .
\end{equation*}
The $k$-th \emph{Wiener chaos} $C_k \equiv C_k(W)$ associated to the isonormal Gaussian process $W$ is defined as the closure
in $L^2(\Omega) \equiv  L^2(\Omega, \mathcal F, \mathbb P)$ of the linear space generated by the random variables 
\begin{equation*}
   \left\lbrace H_k(W(h)) : \ h\in \mH \, , \ \| h\|_{\mH}=1 \right\rbrace.
\end{equation*}
It turns out that the closed subspaces $C_k$ and $C_{k'}$ of $L^2(\Omega)$ are orthogonal whenever $k\ne k'$, and the following orthogonal decomposition holds (see e.g.~\cite[Theorem 1.1.1]{Nua95}): 
\begin{equation*}
    L^2(\Omega) = \bigoplus_{k=0}^{\infty} C_k \, .
\end{equation*}
Equivalently, every $F\in L^2(\Omega)$ can be written as an orthogonal series converging in $L^2(\Omega)$ of the form 
\begin{equation*}
    F = \sum_{k=0}^{\infty} J_k F \, ,
\end{equation*}
where $J_k$ stands for the orthogonal projection operator onto $C_k$. Plainly, we have that $J_0 F = \mathbb E[F]$ and 
\begin{equation*}
    \Var(F) = \sum_{k=1}^{\infty} \mathbb E\!\left[\left(J_k  F\right)^2 \right] .
\end{equation*}
Now let $F\in L^2(\Omega)$ be more regular, namely $F\in \mathbb D^{k,2}$ for some $k\ge 1$; then  
\begin{equation}\label{equality}
\left\| \mathbb E\!\left[D^k F\right] \right\|_{H^{\otimes k}} = \sqrt{k!} \left\| J_k F \right\|_{L^2(\Omega)} .
\end{equation}
(This follows e.g.~from \cite[Corollary 2.7.8]{NP12} and It\^o's isometry \cite[Theorem 2.7.7]{NP12}.) In order to give some intuition on the proof, let us check (\ref{equality}) for $k=2$ and $F\in \mathcal S$: let $F=f(W(h_1), \dots, W(h_m))$ for some smooth function $f:\mathbb R^m \to \mathbb R$ whose derivatives have at most polynomial growth at infinity, and assume w.l.o.g.~that $ \{ h_i \}_{i=1,\ldots,m} $ are orthonormal vectors in $\mathcal H$. Then by definition
\begin{eqnarray*}
D^2 F = \sum_{i,j=1}^m \frac{\partial^2 f}{\partial x_i \partial x_j} (W(h_1), \dots, W(h_m)) \, h_i \otimes h_j \, ,
\end{eqnarray*} 
so that 
\begin{eqnarray*}
\left\| \mathbb E\!\left[D^2 F\right] \right\|^2_{\mathcal H^{\otimes 2}} = \sum_{i,j=1}^m \left( \mathbb E \! \left [\frac{\partial^2 f}{\partial x_i \partial x_j} (W(h_1), \dots, W(h_m)) \right ] \right)^2,
\end{eqnarray*} 
the tensors $h_i\otimes h_j$ being orthonormal in $\mathcal{H}^{\otimes 2}$. By the Gaussian integration-by-parts formula, that is, 
Stein's Lemma \cite[Lemma 3.1.2]{NP12}, we have
\begin{eqnarray*}
\mathbb E \! \left [\frac{\partial^2 f}{\partial x_i \partial x_j} (W(h_1), \dots, W(h_m)) \right ] = \begin{cases}
\mathbb E[F H_2(W(h_i))] \qquad & i=j \, , \\
\mathbb E[F H_1(W(h_i)) H_1(W(h_j))] \qquad & i \ne j \, ,
\end{cases}
\end{eqnarray*}
where $H_1$ (resp.~$H_2$) denotes the first (resp.~second) Hermite polynomial. Hence we can write 
\begin{eqnarray*}
\left\| \mathbb E\!\left[D^2 F\right] \right\|^2_{\mathcal H^{\otimes 2}} = \sum_{ \substack{i,j=1 \\ i\ne j} }^m \left( \mathbb E[F H_1(W(h_i)) H_1(W(h_j))]\right)^2 + \sum_{i=1}^m \left( \mathbb{E}\!\left[F H_2(W(h_i))\right] \right)^2.
\end{eqnarray*} 
Now it suffices to recall that (cf.~\cite[Proposition 1.1.1]{Nua95}) 
\begin{eqnarray*}
J_2(F) &=& \sum_{\substack{i,j=1 \\ i<j}}^m \mathbb E[F H_1(W(h_i)) H_1(W(h_j))] \, H_1(W(h_i)) \, H_1(W(h_j)) \\
&&+ \sum_{i=1}^m \mathbb E \!\left [F \, \frac{H_2(W(h_i))}{\sqrt 2} \right ] \frac{H_2(W(h_i))}{\sqrt 2} \, ,
\end{eqnarray*} 
that implies 
\begin{eqnarray*}
\left\| J_2(F) \right\|^2_{L^2(\Omega)} &=& \sum_{\substack{i,j=1 \\ i<j}}^m \left( \mathbb E[F H_1(W(h_i)) H_1(W(h_j))] \right)^2 + \sum_{i=1}^m \left( \mathbb E\!\left [F  \,\frac{H_2(W(h_i))}{\sqrt 2} \right ]\right)^2\\
&=& \frac12  \sum_{\substack{i=1 \\ i\ne j}}^m \left( \mathbb E[F H_1(W(h_i)) H_1(W(h_j))] \right)^2 + \frac{1}{2} \sum_{i=1}^m \left( \mathbb E \! \left [F H_2(W(h_i)) \right ] \right)^2  \\
&=& \frac12 \left\| \mathbb E\!\left[D^2 F\right] \right\|^2_{\mathcal H^{\otimes 2}} .
\end{eqnarray*} 
Taking the square root on both sides we get (\ref{equality}) for $k=2$ and $F\in \mathcal S$. 

\begin{remark}[Wiener chaos and Poincar\'e constants]\rm
Let $q>2$, $k\in\N_{\ge 1}$ and $F\in C_k$. Since $\mathbb E[F]=J_0(F)=0$, the hypercontractivity property of the Ornstein-Uhlenbeck semigroup gives
\begin{align*}
\|F-\mathbb E[F]\|_{L^q(\Omega)}
= & \|F\|_{L^q(\Omega)}
\leq (q-1)^{\frac k 2} \, \|F\|_{L^2(\Omega)} \, .
\end{align*}
If $F=H_k(W(h))$ for some $h\in \mathcal H$, then, since $H'_k = k H_{k-1}$, we immediately have that 
\begin{equation}\label{bello}
k\|F\|^2_{L^2(\Omega)}=\|J_{k-1}(DF)\|_{L^2(\Omega;\mathcal{H})}^2 \, ;
\end{equation}
 actually, (\ref{bello}) holds for every $F\in C_k$, 
and 
$$
J_{k-1}(DF)=D(J_k(F))=DF \, .
$$ 
Therefore,
\begin{align*}
    \|F\|_{L^2(\Omega)}
    = k^{-\frac{1}{2}}\|DF\|_{L^2(\Omega;\mathcal{H})}
    \leq k^{-\frac 1 2 }\|DF\|_{L^q(\Omega;\mathcal{H})} \, .
\end{align*}
By combining the above computations we thus get 
\begin{equation}\label{pp}
    \|F-\mathbb E[F]\|_{L^q(\Omega)}
    \leq \left[\frac{(q-1)^k}{k}\right]^{\frac{1}{2}}\|DF\|_{L^q(\Omega;\mathcal{H})}
\end{equation}
for every $q\in(2,\infty)$. We have just obtained a Poincar\'e inequality for elements of any Wiener chaos in a very simple way; note moreover that upon choosing $k=1$ in (\ref{pp}) we recover the same constant as in \eqref{eq_poincare1}.
\end{remark}

\section{Proofs of technical results}\label{appendixB}
\begin{proof}[Proof of Lemma \ref{lem:core_L_2}]
Recall that a core is a dense subset w.r.t.~the graph norm $ \| \cdot \|_{L^q(\Omega;\mV)}  + \| \mathcal L _q^\mV(\cdot) \|_{L^q(\Omega;\mV)} $. Since $ \mathsf{D}_\mV(\mathcal{L}_q) \supset \mathcal{S}_\mV $ is dense in $L^q(\Omega;\mV)$, to prove the statement it is enough to show that $P_t^\mV(\mathsf{D}_\mV(\mathcal{L}_q))\subset \mathsf{D}_\mV(\mathcal{L}_q)$ for all $t>0$, by means of a classical double-approximation argument (see e.g.~\cite[Chapter 2, Proposition 1.7]{EnNa00}). To this aim, let $F \in \mathsf{D}_\mV(\mathcal{L}_q)$ be of the form 
$$
F = \sum_{j=1}^J F_j \, v_j \, ,
$$
with $J\in \mathbb N_{\ge 1}$, $F_j \in \mathsf{D}(\mathcal{L}_q)$ and $v_j\in \mV$ for every $ j=1 , \ldots , J $. Since $(P_t^\mV)_{t\geq0}$ is the extension of the linear semigroup $(P_t)_{t\geq0}$, as recalled above, it follows that
$$
P_t^\mV F =  P_t^\mV \! \left(\sum_{i=1}^J F_j \, v_j \right)
= \sum_{j=1}^J P_t^\mV(F_j \, v_j)
=\sum_{j=1}^J P_t(F_j) \, v_j \, .
$$
Because $(P_t)_{t\geq0}$ is a strongly continuous semigroup on $L^q(\Omega)$, we have that $P_t G \in \dom(\lop_q) $ for all $ G \in \dom(\lop_q)$. This means that
\begin{align*}
    \sum_{j=1}^J P_t(F_j) \, v_j \in \dom_\mV(\lop_q) \, ,
\end{align*}
which gives the thesis.
\end{proof}

\smallskip 
\begin{proof}[Proof of Lemma \ref{lemma:core_tilde_L}]
As in the proof of Lemma \ref{lem:core_L_2}, it suffices to show that $T_t^\mV(\dom_\mV(\lop_2))\subset \dom_\mV(\lop_2)$ for all $t>0$. 
First of all,  we claim that $\dom_\mV(\lop_2) \subset \dom(\tilde \lop^\mV) $ and $\lop_2^\mV=\tilde \lop^\mV$ on $\dom_\mV(\lop_2)$. To this aim, let $F \in \dom_\mV(\lop_2) $ be of the form
$$
F = \sum_{j=1}^J F_j \, v_j \, ,
$$
with $J\in \mathbb N_{\ge 1}$, $F_j \in \mathsf{D}(\mathcal{L}_2)$ and $v_j\in \mV$ for every $ j=1 , \ldots , J $. 
Without loss of generality, since $ \dom (\lop_2) $ is a linear space, we may assume that $\{ v_j \}_{j = 1, \ldots, J}$ are orthonormal vectors in $\mV$. Then, recalling Remark \ref{rmk:forma_dir_L_2},
\begin{align*}
\mathcal E^\mV(F,G)
= \sum_{j=1}^J \mathbb E\!\left[\left\langle D F_j , D(\langle G,v_j\rangle_\mV)\right\rangle_{\mathcal H}\right]
= & \, -\sum_{j=1}^J \mathbb E\!\left[\left(\lop_2 F_j\right) \left\langle G,v_j\right\rangle_\mV\right] \\
= & \, - \mathbb E\!\left[\left\langle \sum_{j=1}^J \left(\lop_2 F_j\right) v_j,G\right\rangle_{\!\mV} \right]
\end{align*}
for all $G\in S_\mV $. Hence, we deduce that $ F \in \dom(\tilde \lop^\mV)$ and
\begin{align*}
    \tilde \lop^\mV F = \sum_{j=1}^J \left( \lop_2 F_j \right) v_j
   = \lop_2^\mV F \, ,
\end{align*}
which proves the claim. Given any $\lambda>0$, let $R(\lambda,\lop_2)[\cdot]$ be the resolvent operator in $ L^2(\Omega) $ associated to $ \lop_2 $. Then the function
\begin{align*}
R[F]:=  \sum_{j=1}^J R(\lambda,\lop_2)[F_j]\, v_j    
\end{align*}
belongs to $ \dom_\mV(\lop_2) $ and solves the equation
\begin{align*}
\lambda U-  \tilde \lop^\mV U = F \, .    
\end{align*}
Indeed, by what we have just shown, it holds
\begin{align*}
\tilde \lop^\mV\!\left( R[F] \right)
=  \sum_{j=1}^J \lop_2\!\left( R(\lambda,\lop_2)[F_j] \right) v_j
=  \sum_{j=1}^J \left(\lambda R(\lambda,\lop_2)[F_j] - F_j\right) v_j
=  \lambda R[F] -F \, .
\end{align*}
This implies that $R[\cdot] = R(\lambda,\tilde \lop^\mV)[\cdot]$ on $ \dom_\mV(\lop_2) $, where $ R(\lambda,\tilde \lop^\mV)[\cdot] $ stands for the resolvent operator in $ L^2(\Omega;\mV) $ associated to $ \tilde \lop^\mV $. Since resolvent operators can be represented by means of the Laplace transform of the semigroup (see e.g.~\cite[Chapter 2, Theorem 1.10]{EnNa00}), in particular we have 
\begin{align*}
R(\lambda,\tilde \lop^\mV)[F] & =\int_0^{+\infty} e^{-\lambda t} \, T_t^\mV F \, d t \, , \\   R(\lambda,\lop_2)[F_j] & =\int_0^{+\infty} e^{-\lambda t} \, P_t F_j \, d t \qquad j=1,\ldots,J \, , 
\end{align*}
where the integrals are understood e.g.~in the Bochner sense. Then, if $\{w_j\}_{j \in \N_{\ge 1}}$ is an orthonormal basis for $\mV$ such that $w_j=v_j$ for $j=1,\ldots,J$, we infer that
\begin{align*}
0 = \left\langle R[F]-R(\lambda,\tilde \lop^\mV)[F],w_j \right\rangle_\mV
= \int_0^{+\infty} e^{-\lambda t} \left(P_t F_j - \left\langle T^\mV_t F , w_j\right\rangle_\mV\right) dt 
\end{align*}
if $j=1,\ldots,J$, whereas
\begin{align*}
0= \left\langle R(\lambda,\tilde \lop^\mV)[F] , w_j \right\rangle_\mV
= \int_0^{+\infty} e^{-\lambda t} \left\langle T^\mV_t F , w_j \right\rangle_\mV \, dt 
\end{align*}
if $j>J$. Since the above identities are true for any $ \lambda>0 $, the injectivity of the Laplace transform implies that
$$
\begin{cases}
\left\langle T^\mV_t F , w_j \right\rangle_\mV = P_t F_j & \text{if } j =1,\ldots,J \, , \\
\left\langle T^\mV_t F , w_j \right\rangle_\mV =0 & \text{if } j   > J \, ,
\end{cases}
$$
for all $ t>0 $. This yields
\begin{align*}
T^\mV_t F  = \sum_{j=1}^J P_t F_j \, v_j \in \dom_\mV(\lop_2) \qquad \forall t>0 \, ,
\end{align*}
and the proof is complete.
\end{proof}

\bibliographystyle{alpha}
\bibliography{bibfile.bib}

\begin{thebibliography}{HPA95}

\bibitem[AF03]{AF03}
Robert~A. Adams and John J.~F. Fournier.
\newblock {\em Sobolev spaces}, volume 140 of {\em Pure and Applied Mathematics
  (Amsterdam)}.
\newblock Elsevier/Academic Press, Amsterdam, second edition, 2003.

\bibitem[Che81]{Che81}
Herman Chernoff.
\newblock A note on an inequality involving the normal distribution.
\newblock {\em Ann. Probab.}, 9(3):533--535, 1981.

\bibitem[EN00]{EnNa00}
Klaus-Jochen Engel and Rainer Nagel.
\newblock {\em One-parameter semigroups for linear evolution equations}, volume
  194 of {\em Graduate Texts in Mathematics}.
\newblock Springer-Verlag, New York, 2000.
\newblock With contributions by S. Brendle, M. Campiti, T. Hahn, G. Metafune,
  G. Nickel, D. Pallara, C. Perazzoli, A. Rhandi, S. Romanelli and R.
  Schnaubelt.

\bibitem[FU00]{FeUs00}
D.~Feyel and A.~S. \"{U}st\"{u}nel.
\newblock The notion of convexity and concavity on {W}iener space.
\newblock {\em J. Funct. Anal.}, 176(2):400--428, 2000.

\bibitem[Gun89]{Gun89}
Richard~F. Gundy.
\newblock {\em Some topics in probability and analysis}, volume~70 of {\em CBMS
  Regional Conference Series in Mathematics}.
\newblock Published for the Conference Board of the Mathematical Sciences,
  Washington, DC; by the American Mathematical Society, Providence, RI, 1989.

\bibitem[Hoo81]{Hoo81}
James~G. Hooton.
\newblock Compact {S}obolev imbeddings on finite measure spaces.
\newblock {\em J. Math. Anal. Appl.}, 83(2):570--581, 1981.

\bibitem[HPA95]{HP95}
Christian Houdr\'{e} and V\'{\i}ctor P\'{e}rez-Abreu.
\newblock Covariance identities and inequalities for functionals on {W}iener
  and {P}oisson spaces.
\newblock {\em Ann. Probab.}, 23(1):400--419, 1995.

\bibitem[Maa10]{Maa10}
Jan Maas.
\newblock Malliavin calculus and decoupling inequalities in {B}anach spaces.
\newblock {\em J. Math. Anal. Appl.}, 363(2):383--398, 2010.

\bibitem[Mey84]{Mey84}
P.-A. Meyer.
\newblock Transformations de {R}iesz pour les lois gaussiennes.
\newblock In {\em Seminar on probability, {XVIII}}, volume 1059 of {\em Lecture
  Notes in Math.}, pages 179--193. Springer, Berlin, 1984.

\bibitem[Mil09]{Mil09}
Emanuel Milman.
\newblock On the role of convexity in isoperimetry, spectral gap and
  concentration.
\newblock {\em Invent. Math.}, 177(1):1--43, 2009.

\bibitem[MR92]{MaRo92}
Zhi~Ming Ma and Michael R\"{o}ckner.
\newblock {\em Introduction to the theory of (nonsymmetric) {D}irichlet forms}.
\newblock Universitext. Springer-Verlag, Berlin, 1992.

\bibitem[Nas58]{Nas58}
J.~Nash.
\newblock Continuity of solutions of parabolic and elliptic equations.
\newblock {\em Amer. J. Math.}, 80:931--954, 1958.

\bibitem[NP12]{NP12}
Ivan Nourdin and Giovanni Peccati.
\newblock {\em Normal approximations with {M}alliavin calculus}, volume 192 of
  {\em Cambridge Tracts in Mathematics}.
\newblock Cambridge University Press, Cambridge, 2012.
\newblock From Stein's method to universality.

\bibitem[NPR09]{NPR09}
Ivan Nourdin, Giovanni Peccati, and Gesine Reinert.
\newblock Second order {P}oincar\'{e} inequalities and {CLT}s on {W}iener
  space.
\newblock {\em J. Funct. Anal.}, 257(2):593--609, 2009.

\bibitem[Nua06]{Nua95}
David Nualart.
\newblock {\em The {M}alliavin calculus and related topics}.
\newblock Probability and its Applications (New York). Springer-Verlag, Berlin,
  second edition, 2006.

\bibitem[Pis88]{Pis88}
Gilles Pisier.
\newblock Riesz transforms: a simpler analytic proof of {P}.-{A}. {M}eyer's
  inequality.
\newblock In {\em S\'{e}minaire de {P}robabilit\'{e}s, {XXII}}, volume 1321 of
  {\em Lecture Notes in Math.}, pages 485--501. Springer, Berlin, 1988.

\bibitem[PV14]{PV14}
Matthijs Pronk and Mark Veraar.
\newblock Tools for {M}alliavin calculus in {UMD} {B}anach spaces.
\newblock {\em Potential Anal.}, 40(4):307--344, 2014.

\bibitem[Shi04]{Shi98}
Ichiro Shigekawa.
\newblock {\em Stochastic analysis}, volume 224 of {\em Translations of
  Mathematical Monographs}.
\newblock American Mathematical Society, Providence, RI, 2004.
\newblock Translated from the 1998 Japanese original by the author, Iwanami
  Series in Modern Mathematics.

\bibitem[Sug85]{Sug85}
Hiroshi Sugita.
\newblock Sobolev spaces of {W}iener functionals and {M}alliavin's calculus.
\newblock {\em J. Math. Kyoto Univ.}, 25(1):31--48, 1985.

\bibitem[Tag09]{ta08}
Robert~J. Taggart.
\newblock Pointwise convergence for semigroups in vector-valued {$L^p$} spaces.
\newblock {\em Math. Z.}, 261(4):933--949, 2009.

\bibitem[vN15]{VN15}
Jan van Neerven.
\newblock The {$L^p$}-{P}oincar\'{e} inequality for analytic
  {O}rnstein-{U}hlenbeck semigroups.
\newblock In {\em Operator semigroups meet complex analysis, harmonic analysis
  and mathematical physics}, volume 250 of {\em Oper. Theory Adv. Appl.}, pages
  353--368. Birkh\"{a}user/Springer, Cham, 2015.

\end{thebibliography}

\end{document}